\documentclass[11pt, two side,emlines]{amsart}
\usepackage{amssymb,latexsym,xy,eucal,mathrsfs, graphicx, tikz}
\textwidth=15.2cm \textheight=24cm \theoremstyle{plain}
\newtheorem{lemma}{Lemma}[section]
\newtheorem{theorem}[lemma]{Theorem}
\newtheorem{proposition}[lemma]{Proposition}
\newtheorem{corollary}[lemma]{Corollary}

\theoremstyle{definition}
\newtheorem{example}[lemma]{Example}
\newtheorem{remark}[lemma]{Remark}

\newtheorem{definition}[lemma]{Definition}

\numberwithin{equation}{section}  \voffset
-55truept \hoffset -15truept
\begin{document}
\baselineskip 15truept
\title[Spectra of the zero-divisor graph of finite rings]{Spectra of the zero-divisor graph of finite rings}

\author[K.D. Masalkar]%
{Krishnat D. Masalkar$^{A}$}	
\address{\rm $A$ Department of Mathematics, Abasaheb Garware College, Pune-411 004, India.}
\email{\emph{krishnatmasalkar@gmail.com}}
\author[A.S. Khairnar]%
{Anil Khairnar$^{A,1}$}	
\address{\rm $A$ Department of Mathematics, Abasaheb Garware College, Pune-411 004, India.}
\email{\emph{anil\_maths2004@yahoo.com, ask.agc@mespune.in}}
\author[A.M. Lande]%
{Anita Lande$^{A}$}	
\address{\rm $A$ Department of Mathematics, Abasaheb Garware College, Pune-411 004, India.}
\email{\emph{anita7784@gmail.com, abw.agc@mespune.in}}
\author[A.A. Patil]%
{Avinash Patil$^{B}$}	
\address{\rm $B$ Department of Mathematics, JET's Z. B. Patil College, Dhule-424 002, India.}
\email{\emph{avipmj@gmail.com}}
 \maketitle 
 \footnotetext[1]{Corresponding author}
 \begin{abstract} The zero-divisor graph $\Gamma(R)$ of a ring $R$ is a graph with nonzero zero-divisors of $R$ as vertices and distinct vertices $x,y$ are adjacent if $xy=0$ or $yx=0$. We provide an equivalence relation on a ring $R$ and express $\Gamma(R)$ as a generalized join of graphs on equivalence classes of this relation. We determined  the adjacency and Lapalcian spectra of   $\Gamma(R)$ when $R$ is a finite semisimple ring. 
\end{abstract}
\section{Introduction}

Algebra and graph contribute significant applications in the
development of artificial intelligence, information systems,
image processing, clustering analysis, medical diagnosis and
decision making.    
Graph theory that can be used to describe the relationships
among several individuals has numerous applications
in diverse fields such as modern sciences and technology,
database theory, data mining, neural networks, expert
systems, cluster analysis, control theory, and image capturing.  
 
 Diagonalization of matrices is one of the techniques in mathematics. Most of the time diagonalization is discussed for real or complex matrices. A large part of linear algebra can be performed over arbitrary commutative rings, and also over non-commutative rings. It is therefore natural to ask how the theory can be extended from the real or complex case to arbitrary rings. In \cite{Dan} Dan Laksov propose a method for diagonalization of matrices with entries in commutative rings.

 Let $G=\left<V,~ E\right>$ be a simple undirected graph with a vertex set $V$ and an edge set $E$. The cardinality of $V$ is the \textit{order} of $G$. If there is an edge $e\in E$ with end vertices $u$ and $v$ then we say that $u$ and $v$ are \textit{adjacent} and the edge $e$ is denoted by $u-v$. For any vertex $u$ in $G$, $N(u)=\left\{v\in V(G) \colon u-v \in E(G)\right\}$ is the \textit{neighborhood} of $u$ and  $d(u)=|N(u)|$ is a \textit{degree} of $u$. A graph $G$ is $r-$\textit{regular} if every vertex has the same degree equal to $r$.The notion of the compressed graph is useful in studying the properties of graphs.
     The relation  $\approx$  (which is an equivalence relation) on a vertex set $V$ is defined  by  $a\approx b~~\text{ if and only if } ~~N(a)=N(b).$
   Let $\displaystyle \frac{V}{\approx}=\{[a] \colon [a]=\{b\in V \colon b\approx a \} \}$ be set of its equivalence classes.
   The \textit{compressed graph} $G^{\approx}$ is a graph on $\frac{V}{\approx}$ such that  $[a]^{\approx}-[b]^{\approx}$ is an edge if and only if $a-b$ is an edge in $G$.    
  \par The \textit{adjacency matrix} and the \textit{Laplacian matrix}  of a graph $G=\left<V=\{1,2,...,n\}, E \right>$ are   given by $A(G)=[a_{ij}]_{n\times n}$ and $L(G)=d(G)-A(G)$, where  $a_{ij} = 1$ if $i-j\in E(G)$ and $a_{ij}= 0$ otherwise and $d(G)=diag(d(1), \ldots, d(n))$. A multiset of eigenvalues, $\sigma_A(G) =\left\{\lambda_1^{(s_1)},\ldots, \lambda_n^{(s_n)}\right\}$ of $A(G)$ is the \textit{adjacency spectra }of $G$. The \textit{Laplacian spectra} $\sigma_L(G)$ of a graph $G$ is defined as the multiset of eigenvalues of $L(G)$.   The author refers to \cite{1} for introduction to graph theory and spectral graph theory. The \emph{generalized join of the family of graphs} is defined as below, which is useful to find 
 $\sigma_A(G)$ and $\sigma_L(G)$ of a graph $G$.
 \begin{definition}[{\cite[Definition 2.1]{2}}]
  Let $H=\left<I=\{1,2..n\}, E\right>$ be a graph.
  and $\displaystyle \mathcal{F}=\{G_i=(V_i, E_i)\colon i\in I\}$  be a family of graphs and $V_i\cap V_j=\emptyset$ for all $i\neq j$.  The $H$-\textit{generalized join} of the family $\mathcal{F}$ is denoted by  $\displaystyle{ \bigvee_{H}\mathcal{F}}$ and is a graph formed by replacing each vertex $i$ of $H$ by the graph $G_i$ and joining each vertex of $G_i$ to every  vertex of $G_j$ whenever $i$ and $j$ are adjacent in H.
  \end{definition}
 Motivated from Theorem \ref{t1}, in \cite{2} Cardoso et al.  gave  adjacency spectrum $\displaystyle A\left(\displaystyle\bigvee_{H}\mathcal {F} \right)$  and Laplacian spectrum  $\displaystyle L\left(\bigvee_{H}\mathcal {F}\right)$.
 For sake of convenience, we state result by Fiedler.
  \begin{theorem}[{\cite[Fiedler's result]{7}}]\label{t1}
  Let $A$ be a $m\times m$  symmetric matrix with eigenvalues $\alpha_1, \alpha_2, \ldots, \alpha_m$. Let $u$ be a unit eigenvector of $A$ corresponding to $\alpha_1$.  Let $B$ be another $n\times n$ symmetric matrix with eigenvalues $\beta_1,\beta_2, \ldots, \beta_n$ and $v$ be unit eigenvector of $B$ corresponding to $\beta_1$. Then for any $\rho$ the matrix 
  $C=\begin{bmatrix}
  A& \rho uv^t\\ \rho vu^t & B
  \end{bmatrix}$
  has eigenvalues $\alpha_2, \ldots, \alpha_m,\beta_2, \ldots, \beta_n, \gamma_1, \gamma_2$  where $\gamma_1, \gamma_2$ are eigenvalues of the matrix $C_1=\begin{bmatrix}\alpha_1 &\rho\\ \rho & \beta_1 \end{bmatrix}$.
  \end{theorem}
  Let $R$ be a  ring and $Z(R)$ denote its set of nonzero zero-divisors. Anderson et al. \cite{5} introduced the zero-divisor graph $\Gamma(R)$ of a commutative ring $R$, which was extended to non-commutative rings by Redmond \cite{spr} as the graph with vertex set $Z(R)$ where two vertices
   $a,~ b$ are adjacent if and only if $ab=0~~\text{or}~ba=0$. The aim of considering these graphs is to study the interplay between graph theoretic properties of $\Gamma(R)$ and the algebraic properties of the ring $R$. In (\cite{khairnar2016zero}), the authors examine  preservation of diameter and girth of the zero-divisor graph under extension to Laurent polynomial and Laurent power series rings.
   
   Recently, Chattopadhyay et al. \cite{ch} studied the Laplacian eigenvalues of $\Gamma(\mathbb Z_n)$. Afkhami et al. \cite{Afk} studied the signless Laplacian and normalized Laplacian spectra of $\Gamma(\mathbb Z_n)$.  Bajaj and Panigrahi \cite{Baj} studied the adjacency spectrum of $\Gamma(\mathbb Z_n)$. Pirzada et al. \cite{Pir} studied the adjacency spectrum of $\mathbb Z_{p^Mq^N}$. In \cite{Baj1} Bajaj and Panigrahi studied the universal adjacency spectrum of $\Gamma(\mathbb Z_n)$.   
    Katja M$\ddot {o}$nius \cite{mo} determined adjacency spectrum of $\Gamma \left(\mathbb Z_p \times \mathbb Z_p \times \mathbb Z_p \right)$ and $\Gamma \left(\mathbb Z_p \times \mathbb Z_p \times \mathbb Z_p \times \mathbb Z_p\right)$ for a prime number $p$. Jitsupat Rattanakangwanwong and Yotsanan Meemark \cite{ji} studied the eigenvalues and eigenvectors of adjacency matrix of the zero divisor graphs of finite direct products of finite chain rings.    
   
 \par In this paper, we provide an equivalence relation  $\sim $ on a finite ring $R$ and express $\Gamma(R)$ as $\Gamma(R)^{\sim}-$generalized join of null and complete graphs. By using the equivalence relation $\approx$, $\Gamma(R)$ is expressed as $\Gamma(R)^{\approx}-$generalized join of a family of null graphs. Using Cardoso's result we find the adjacency and Laplacian spectra of $\Gamma(R)$ when $R$ is a finite semisimple ring. Also, we provide a method to find adjacency spectra of a graph which generalized join graph of a family of null graphs. 
 
    \section{representation of zero-divisor graph of rings using generalized join}
   In order to simplify the representation of $\Gamma(R)$, it is often useful to consider the notion called compressed zero-divisor graphs and the notion of the generalized join of graphs. In (\cite{3}), Mulay introduced  a compressed  zero divisor graph of a commutative ring $R$. If $R$ is a commutative ring then the relation $\sim_m$ on $Z(R)$ defined by $a\sim_m b$ if and only if $ann(a)=ann(b)$.
   For a  commutative ring $R$, a \textit{compressed zero-divisor graph} $\Gamma_E(R)$ is a graph with vertex set $\left\{[a]^{\sim_{m}} ~|a\in Z(R)\right\}$, where $~ [a]^{\sim_{m}}=\left\{ x\in Z(R)~|~ ann(x)=ann(a)\right\} $  is equivalence class of the relation $\sim_{m}$ containing $a$ and any two vertices $[a], [b]$ in $\Gamma_E(R)$ are adjacent if and only if $a$ and $b$ are adjacent in $\Gamma(R)$. This notion of compressed zero divisor graph $\Gamma_E(R)$ can be extended to noncmmutative ring. If $R$ be noncommutative ring then for $a\in R$, set of annihilators of $x$ is denoted by $ann(a)$ and it given by
   $ann(a)=\left\{x\in Z(R) ~|~ax=0~~\text{or}~~xa=0\right\}$.Note that $ann(a)=ann_l(a)\cup ann_r(a)$, where $ann_l(a)=\left\{x\in Z(R) ~|~xa=0 \right\}$ and $ann(a)=\left\{x\in Z(R) ~|~ax=0\right\}$. The relation $\sim_m$ is also an equivalence relation on $Z(R)$ when $R$ is a noncommuatative ring. Also for a ring $R$, $\Gamma^{\approx}(R)$ is one of the compressed zero divisor graph with vertex set $\left\{ [a]^{\approx}~|~ a\in  Z(R)\right\}$, where $[a]^{\approx}=\left\{x\in Z(R)~|~ N(x)=N(a)~~\text{in}~~\Gamma(R)\right\}$. Clearly for $a\in \Gamma(R)$, $N(a)=ann(a)\setminus \{a\}.$\\
   Consider ring $R=Z_{18}$. The vertex set of the graph $\Gamma_E(R)$   is $$\left\{[2]^{\sim_m}=\{2,4,8,10,14,16\}, [3]^{\sim_m}=\{ 3,15\}, [6]^{\sim_m}=\{ 6,12\}, [9]^{\sim_m}=\{9\} \right\}$$ while vertex set of the graph $\Gamma(R)^{\approx}$ is $$\left\{[2]^{\approx}=\{2,4,8,10,14,16\}, [3]^{\approx}=\{ 3,15\}, [6]^{\approx}=\{ 6\},[12]^{\approx}=\{12\}, [9]^{\approx}=\{9\} \right\}.$$  
 Let $R$ be a ring then  we will show that, if $R$ is reduced then  $\Gamma_E(R)=\Gamma^{\approx}(R)$. But converse is not true. Ring $Z_4$ is not reduced and $\Gamma_E(Z_4)=\Gamma^{\approx}(Z_4)=K_1$.
  \begin{proposition}
  Let $R$ be a ring. Then $R$ is reduced then $\Gamma_E(R)=\Gamma^{\approx}(R)$.
  \end{proposition}
  \begin{proof}
 Assume $R$ is a reduced ring. Therefore $a^2=0$ imply $a=0$ for any $a\in R$.
 Hence for any $a\in Z(R)$, $ann(a)=ann(a)\setminus\{a\}=N(a)$. So for any $a,b\in R$, $ann(a)=ann(b)$ if and only if $N(a)=N(b)$. Therefore $a\sim_m b$ if and only if $a\approx b$. This imply $[a]^{\sim_m}=[a]^{\approx}$ for any $a\in Z(R)$. Hence $\Gamma_E(R)=\Gamma^{\approx}(R)$.
  \end{proof}
  In following proposition we give the relation between equivalence classes of  relations $\sim_m$ and $\approx$ defined on the commutative ring with unity.
  \begin{proposition}\label{prop2.2}
 Let $R$ be a commutative ring with unity 1 and $a\in Z(R)$. If $R$ contains unit $u$ with $(1-u)^2\neq 0$ then  
  \begin{enumerate}
  \item  $a^2\neq 0$ imply  $[a]^{\approx}=[a]^{\sim_m}$.
  \item $a^2=0$ imply  $[a]^{\approx}=\{a\}$.  
  \end{enumerate}
  \end{proposition} 
  \begin{proof}
  Let $R$ be commutative ring with unity 1 and $u$ is unit in $R$ with $(1-u)^2\neq 0$. We will prove statement (1).
  Let $a\in Z(R)$ and $a^2\neq 0$. 
  Let $x\in [a]^{\sim_m}$ . Then $ann(x)=ann(a)$,
  and hence $(x)=\frac{R}{ann(x)}=\frac{R}{ann(a)}=(a)$. Therefore $a=xc$ for some $c\in R$. Since $a^2\neq 0$, we have $x^2\neq 0$. Therefore $N(x)=ann(x)=ann(a)=N(a)$ . Hence $x\in [a]^{\approx}$. This gives $[a]^{\sim}\subseteq [a]^{\approx}$.
  Let $x\in[a]^{\approx}$. Then $N(x)=ann(x)\setminus \{x\}=N(a)=ann(a)$. Hence $ax\neq 0$. If $x^2=0$  then $xu\in N(x)=N(a)$. This implies that $axu=0$ and hence $ax=0$, which contradicts to $ax \neq 0$. Therefore $x^2\neq 0$. This yields $ann(x)=ann(x)\setminus\{x\}=ann(a)$. This gives $x\in [a]^{\sim_m}$. Therefore $[a]^{\approx}\subseteq [a]^{\sim_m}$. Thus $[a]^{\sim_m}=[a]^{\approx}$.\\
  Now, we will prove statement (2). Let $a^2=0$. If $x\in [a]^{\sim_m}$ then  $(a)=\frac{R}{ann(a)}=\frac{R}{ann(x)}=(x)$. Hence $x^2=xa=0$, that gives $x\in N(a)\setminus N(x)$. This implies that $x\notin [a]^{\approx}$.
  If $x\notin [a]^{\sim_m}$ then we will show that $x\notin [a]^{\approx}$. If $x=au$ then $x\in N(a)\setminus N(x)$ and hence $x\notin [a]^{\approx}$. Suppose that $x\neq au$. If $x\in [a]^{\approx}$, then $N(a)=N(x)$ and hence $xa\neq 0$. Since $au\in N(a)=N(x)$, therefore $xau=0$. Hence $xa=0$, which is a contradiction. Thus $x\notin [a]^{\approx}$.
  \end{proof}
  In the following proposition we give the relation between equivalence classes of  relations $\sim_m$ and $\approx$ defined on noncommutative ring with unity.
    \begin{proposition}\label{prop2.2}
   Let $R$ be a noncommutative ring with unity 1 and $a\in Z(R)$. If there exist units $u$ and $v$ in $R$ such that $u+v=1$ then  
    \begin{enumerate}
    	\item $a^2=0$ imply  $[a]^{\approx}=\{a\}$.  
    \item  $a^2\neq 0$ imply  $[a]^{\approx}=[a]^{\sim_m}$.
     
    \end{enumerate}
    \end{proposition}
  \begin{proof}
  Let $R$ be a non-commutative ring with unity $1$ and $a\in Z(R)$.\\
$(1):$ Let $a^2=0$.
   Let $x\in [a]^{\sim_m}$ and $x\neq a$. Therefore $a\in ann(a)=ann(x)$. Therefore $ax=0$ or $xa=0$. This gives $x\in N(a)\setminus N(x)$. So $x\notin [a]^{\approx}$.
    Let $x\notin [a]^{\sim_m}$. Assume contrary $x\in [a]^{\approx}$. Therefore $xa\neq 0$ and $ax\neq 0$.
    Since $1-u$ and $u$ are units, $au\neq a$ and $a(1-u)\neq a$. Since $N(x)=N(a)$, $ax=aux+a(1-u)x=0+0=0$. Which is contradiction.
    Therefore $x\notin [a]^{\approx}$. Hence we conclude that $[a]^{\approx}=\{a\}$.\\  
   $(2):$ Let $a^2\neq 0$, $x\in [a]^{\sim_m}$ and $x\neq a$. If $x^2\neq 0$ then $N(x)=ann(x)=ann(a)=N(a)$. Hence $x\in [a]^{\approx}$.
  Assume that $x^2=0$. Since $x\in ann(x)=ann(a)$, $x\in N(a)\setminus N(x)$.Therefore  $x\notin [a]^{\approx}$.
  Let $y\notin [a]^{\sim_m}$.  If  $ya=0$ or $ay=0$ then $y\in N(a)\setminus N(y)$ and hence $y\notin [a]^{\approx}$. Therefore assume that $ya\neq 0$ and $ay\neq 0$. Let $y^2\neq 0$. If $y\in [a]^{\approx}$ then $ann(y)=ann(y)\setminus\{y\}=N(y)=N(a)=ann(a)\setminus \{a\}=ann(a)$. So $y\in [a]^{\sim_m}$. This contradicts to fact that $y\notin [a]^{\sim_m}$. 
  If $y^2=0$ then $yu\neq y$ and $y(1-u)\neq y$, as $1-u$ and $u$ are units. If $y\in [a]^{\approx}$ then $yu,~~ y(1-u)\in N(y)=N(a)$. Hence $yua=0$ and $y(1-u)a=0$. This implies that $ya=y(1-u)a+yua=0+0=0$. This contradicts to fact that $ya\neq 0$. Therefore $y\notin [a]^{\approx}$.
  Hence we conclude that, $[a]^{\approx}=[a]^{\sim_m}\setminus N_2$.
 From (1), we get that $[a]^{\approx}=[a]^{\sim_m}$.
  \end{proof} 
     
 The following proposition gives another equivalence relation $\sim$  on a ring with unity.
  \begin{proposition}
  Let $R$ be a ring with unity. A binary relation  $\sim $ on $Z(R)$ defined by $$a\sim b~ \text{if and only if}~ a=ub=bv,~\text{ for some units } u,~ v\in R,$$ 
  is an equivalence relation.
  \end{proposition}
  \begin{proof}
 Let $x,y,z\in Z(R)$. Since $x=1x=x1, $   $x\sim x$.
 Also $x\sim y$ implies $x=uy=yv$, for some units $u,v\in R,$ which gives $y=u^{-1}x=xv^{-1}$ and hence  $y\sim x$.
 If $x \sim y$ and $y\sim z$, then there exist units $u_1, u_2, v_1, v_2$ such that $y=u_1x=xv_1$ and $z=u_2y=yv_2$; and so $z=u_2u_1x=v=xv_2v_1$, where $u_2u_1$ and $v_2v_1$ units in $R$. Hence $x\sim z$. Therefore $\sim$ is an equivalence relation on $Z(R)$. 
  \end{proof}
  \begin{corollary}
   Let $R$ be a commutative ring with unity. A binary relation  $\sim $ on $Z(R)$ defined by $$a\sim b~ \text{ if and only if }~ a=ub,~\text{ for some unit } u\in R$$ 
    is an equivalence relation.
  \end{corollary}
 \begin{proposition}
  Let $R$ be a ring and $a,b\in Z(R)$. If  $R$ is finite, reduced, commutative, and has  unity  then $a\sim b $,  $a\approx b$ and $a\sim_m b$ are equivalent.  
 \end{proposition}
 \begin{proof}
 Let $R$ be finite commutative reduced ring with unity. Therefore $R=F_1\times F_2\times ...\times F_k$ , where $F_1,F_2,...F_k$ are finite fields. Let $a=(a_1,a_2,...a_k)\in R$ and $b=(b_1,b_2,...b_n)$ in $R$.
 Assume $a\approx b$. Hence $ann(a)=N(a)=N(b)=ann(b)$. Then $a_i\neq 0$ if and only if $b_i\neq 0$. Therefore there are units $u_i\in F_i$ such that $a_i=u_ib_i$ for all $i=1,2,...,k$.
 Therefore $a=ub$ with $u=(u_1,u_2,...u_k)$ is unit in $R$. Clearly $a\sim b$ then $a=ub$ for some unit in $R$. Therefore $N(a)=ann(a)=ann(b)=N(b)$. Hence $a\sim b$. 
 Also that by proposition (\ref{prop2.2}), $a\approx b$ and $a\sim_m b$ are equivalent.
 \end{proof}
\begin{example}
Let $R$ be a ring with unity. 
\begin{enumerate}
\item Consider ring $R=\mathbb{Z}_{16}$. Then the set of all  zero divisors in $R$ is $Z(R)=\{ 2,4,6,8,10,12,14\}$,  and set of all units in $R$ are $U(R)=\{ 1,3,5,7,9,11,13,15\}$.\\ Equivalence classes with respect to $\sim$ are $$\big\{ \{2,6,10,14\},\{8\},\{4,12 \}\big\},$$ while equivalence classes with respect to $\approx$ are $$\big\{\{2,6,10,14\},\{8\},\{4\}, \{12 \} \big\}.$$ 
\item Consider matrix ring $M_n(F)$ over finite field $F$.
Let $A\in M_n(F)$ and $B\in [A]^{\sim}$. Then $A^2=0$ 
if and only if $B^2=AB=BA=0$. Since $M_n(F)$ has unit $u$ such that $1-u$ is also unit, $[A]^{\approx}=\{A\}$ if $A^2=0$. Also $[A]^{\sim}\subseteq [A]^{\approx}=[A]^{\sim_m}$ if $A^2\neq 0$.

\end{enumerate}
\end{example}
 Following relation given in (\cite{ch}), equivalence relation on defined on  ring $Z_n$.\\
  $a\sim_1 b$ in $\mathbb{Z}_n \text{ if and only if } (a,n)=(b,n)$, where $(a,n)$ is the gcd of $a$ and $n$.\\
 \begin{proposition}\label{p5}
Let $a,b $ in $Z_n$. Then $a\sim b$ is equivalent to $a\sim_1 b$.
\end{proposition}
\begin{proof}
We prove that $a\sim_i b$ if and only if $a\sim b$, for $i=1,2,3,4$.\\
\textit{Claim (1)}: $a\sim_1 b$ if and only if $a\sim b$.\\
 Assume that $a\sim_1 b$ in $\mathbb{Z}_n$. Suppose $(a,n)=(b,n)=d.$  Hence $ann(a)=ann(b)=(n/d)$. Then $(a)=\frac{Z_n}{ann(a)}=\frac{Z_n}{ann(b)}=(b)$. Assume $n=p_1^{k_1}p_2^{k_2}...p_m^{k_m}$ is prime factorization of $n$.\\ By chineese remainder theorem,
 $$Z_n=\frac{Z_n}{(p_1^{k_1})}\times \frac{Z_n}{(p_2^{k_2})}\times...\times \frac{Z_n}{(p_m^{k_m})} $$ and the ismorphism is given by $\phi(x)=(x+(p_1^{k_1}),....,x+(p_m^{k_m})).$\\
Let $a=a_i+(p_i^{k_i})$ for all $i=1,2,...m$. We prove that $(a_i)=(b_i)$ in $\frac{Z_n}{(p_i^{k_i})}$ for each $i=1,2...,m$. Since $(a)=(b)$ in $Z_n$, there exist $c\in Z_n$ such that $a=bc$. Applying isomorphism $\phi$, we get  $\phi(a)=\phi(b)\phi(c)$. Hence $(\phi(a))\subseteq (\phi(b))$. Therefore $(a_i)\subseteq (b_i)$ for all $i=1,2...,m$. Similarly we can show that $(b_i)\subseteq (a_i).$ Therefore we get $(a_i)=(b_i)$ in each of ring $\frac{Z_n}{(p_i^{k_i})}$ and $a=a_1a_2..a_m,~~b=b_1b_2...b_m$. 
 In ring $\frac{Z_n}{(p_i^{k_i})}$,  there exist unit $u_i$ such that $a_i=b_iu_i$. Hence we get
 $a=a_1a_2...a_m=b_1b_2...b_m(u_1u_2...u_m)=bu$, where $u=u_1u_2...u_m$ is an unit in $Z_n$. So $a\sim b$.\\ 
 Conversely, assume that $a\sim b$ then there is an unit $ u ~\text{such that}~ a=ub$. This yields $(a,n)=(ub, n)=(b,n)$, that is $a\sim_1 b$.
 \end{proof}
\begin{proposition}
Let $F$ be a field. Let $A,B $ in $M_n(F)$. Then $A\sim B$ is equivalent to $column~space(A)=column~space(B)~\text{ and }~\\row~space(A)=row~space(B)$
\end{proposition}
\begin{proof}
 Assume that $A\sim_2 B$ in $Z(M_n(F))$.
Therefore $row~ space(A)=row~ space(B)$  and\\ $column~ space(A)=column~space(B)$. Let $E, F$  be row reduced echelon forms of $A$ and $B$ respectively. Then there exist invertible matrices $C$ and $D$ such that $CA=E,~DB=F$. Since row spaces of $A$ and $B$ are same, we must have $E=F$, which imply $CA=DB$, {\it i.e.}, $A=PB\text{ and }P=C^{-1}D$. Similarly, there exists an invertible matrix $Q$ such that $A=BQ$. Therefore $A\sim B$.\\
 Conversely, assume that $A\sim B$. Hence there exist invertible matrices $P$ and $Q$ such that $A=PB=BQ$. Since $P$ is invertible, there exist elementary matrices, say $E_1, E_2,\ldots, E_k$ such that $P=E_1E_2\ldots E_k$. Also, we know that for any elementary matrix $E,$ $row~space(B)=row~space(EB)$. Hence inductively we get $row~ space(B)=row~ space(PB)=row~ space(A)$. Similarly, we have, $column~ space(A)=column~space(B)$. Thus $A\sim_2 B$.
\end{proof}
 Let $F$ be field and $A, B\in M_n(F)$. Then each of the following statements is equivalent to the statement $A\sim B$. 
   \begin{enumerate}
  \item  $~row~ null~ space(A)=row~ null~ space(B)~and~\\column~null~space(A)=column~null~space(B)$.
  \item   $ ~row~ null~ space(A)=row~ null~ space(B)~\\and~column~space(A)=column~space(B)$. 
  \end{enumerate}
  Now we show that  two relations $\sim$ and $\sim_m$ on a ring are same on  a matrix ring over finite field.
  \begin{proposition}
  Let $R=M_n(F)$  be a matrix ring over field $F$ and $A, B\in R$. Then $A\sim B$ if and only if $A\sim_m B$.
  \end{proposition}
  \begin{proof}
  Let $A\sim B$. Then there are units $U$ and $V$ in $R$ such that  
  $B=UA=AV$. Therefore $CA=0$ if and only if $CB=0$. Also $AC=0$ if and only if $BC=0$. Hence $ann(A)=ann_l(R)\cup ann_r(A)=ann_l(B)\cup ann_r(B)=ann(B)$. Therefore $A\sim_m B$.\\
  Conversely  assume $A\sim_m B$. Therefore $ann_r(A)\cup ann_l(A)=ann_r(B)\cup ann_l(B)$.  Let $E$ and $F$ be idempotents such that  $ann_r(A)=ann_r(E)$ and $ann_l(A)=ann_l(F)$. Note that idempotent $E$ can be obtained from row reduced echelon form $E_A$ of $A$ by arraging leading 1's on diagonal using rwo operations on $E_A$. Let $G$ and $H$ idempotents are such that  Hence $ann_r(B)=ann_r(G)$  and $ann_l(H)=ann_l(H)$. Therefore $ann_r(E)\cup ann_l(F)=ann_r(G)\cup ann_l(H)$. If $E\notin ann_r(G)$ then 
  $ann_r(E)=ann_l(H)=J ~(say)$. Hence  $J$ is a proper two sided ideal of $M_n(F)$ and hence $J=0$. Therefore $E^{-1}$ exist and hence $A^{-1}$ exist, a contradiction. Hence $E\in ann_r(G)$. Similarly $G\in ann_r(E)$. Hence
   $ann_r(A)=ann_r(E)=ann_r(G)=ann_r(B)$. Therefore there exist invertible matrix $P\in M_n(F)$ such that $PA=B$. Similarly there is invertible matrix $Q\in M_n(F)$ such that $AQ=B$.     
  \end{proof} 
  \begin{corollary}
  Let $R$ is a finite semisimple ring and $a, b$ in $R$. Then  $a\sim b$  if and only if $a\sim_m b$.
  \end{corollary}
  \begin{proof}
  Since $R$ is a finite semisimple ring , it is finite direct sum over finite fields. Let $a=A_1\oplus A_2\oplus ...\oplus A_k$ and $b=B_1\oplus B_2\oplus ...\oplus B_k$ . Therefore $a\sim b$ if and only if $A_i\sim B_i$ for all $i=1,2,...k$ if and only if $a\sim_m b$.
  \end{proof} 
Let $R$ be a ring with unity. Let $\frac{Z(R)}{\sim}=\{ [x] \colon [x]=\{y\in Z(R) \colon y\sim x\}\}$ be the set of equivalence classes of $\sim$. Let $\Gamma([x])$ is an induced subgraph of $\Gamma(R)$ on $[x]$, where $[x]\in \frac{Z(R)}{\sim}$. Let $\Gamma(R)^{\sim}$ be a graph on $\frac{Z(R)}{\sim}$ such that $[x]-[y]$ is an edge in $\Gamma(R)^{\sim}$ if and only if $x-y$ is an edge in $\Gamma(R)$. We can write $\Gamma(R)$ as $\Gamma(R)^{\sim}-$ generalized join of family of its induced subgraphs on equivalence classes of $\sim$.
\begin{proposition}\label{p2}
Let $R$ be a ring with unity. Let $\mathcal{F}=\left\{\Gamma([x])\colon [x]\in \frac{Z(R)}{\sim}\right\}$. Then
\begin{enumerate}
\item $\displaystyle \Gamma(R)=\bigvee_{\Gamma(R)^{\sim}} \mathcal{F}.$
\item If $x^2=0$,  then $\Gamma([x])$ is a complete graph. otherwise, it is a null graph.
\item Let $x\in R$ and $e, f\in \Gamma([x])$ with $e^2=e, f^2=f$ then $e=f$ and $\Gamma([x])$ is a null graph.
\item The ring $R$ is reduced ({\it i.e.}, $0$ is the only nilpotent element in $R$) if and only if each graph $\Gamma([x])$ is a null graph.
\end{enumerate}
\end{proposition}
\begin{proof}
\textit{Claim (1)}: Let $x, y\in Z(R)$, and $a\in [x],~b \in [y]$. So there are units $u_1, v_1, u_2, v_2$ such that  $a=u_1x=xv_1$ and  $b=u_2y=yv_2$. Hence
$ab=u_1xyv_2~~\text{and}~~ba=u_2yxv_1$. Therefore $xy=0~~\text{if and only if}~~ ab=0~~\text{and}~~yx=0~~\text{if and only if}~~ ba=0$.\\
Therefore  $[x],[y]~\text{ are adjacent} ~~\text{if and only if}~~ xy=0~\text{ or}~~ yx=0~~\text{if and only if}~~ab=0~~\text{ or}~~ba=0~~\text{if and only if}~~ a, b~\text{are adjacent}. $
Thus, each vertex of $\Gamma([x])$ is adjacent to every vertex of $\Gamma([y])$ if and only if $[x]$ and $[y]$ are adjacent in $\Gamma(R)^{\sim}$.\\
\textit{Claim(2)}: Let $x\in Z(R)$ be fixed. If $a,b \in [x]$, then there exist units $u_1, v_1, u_2, v_2$ such that  $a=u_1x=xv_1$ and  $b=u_2x=xv_2$. Hence
$ab=u_1x^2v_2=0~~\text{or}~~ba=u_2x^2v_1=0$ if and only if $x^2=0$. So  all vertices in $\Gamma([x])$ are adjacent to each other if and only if $x^2=0$.   Therefore $\Gamma([x])$ is either a complete graph or a null graph.\\
\textit{Claim(3)}: If $e,f$ are nonzero idempotents in $\Gamma([x])$, then $e=xu_1=v_1x, f=xu_2=v_2x$, for some units $u_1,u_2,v_1,v_2$ in $R$. Therefore $e=xu_1=xu_2u_2^{-1}u_1=fu_2^{-1}u_1=fu,$ where $u=u_2^{-1}u_1$. Similarly $e=vf,$ where $v=v_1v_2^{-1}$. Hence $fe=f^2u=fu=e$ and $ef=vf^2=vf=e$. Therefore $e=ef=fe$. Similarly we can show that $f=ef=fe$. Hence we get $e=f$.\\
\textit{Claim (4)}: If the ring $R$ is not reduced, then there exists a nonzero  element $y$ such that $y^{2n}=0$ and $y^{2n-1}\neq 0$, for some positive integer $n$. Let $x=y^{n}$ then $x\neq 0$ and $x^2=0$. Therefore by  Claim (2),  $\Gamma([x])$ is a complete graph.   Therefore, if $\Gamma([x])$  is a null graph for each $x\in Z(R)$ then $R$ is a reduced ring.  Conversely, assume that every $\Gamma([x])$ is a null graph. Then $x^2\neq 0$, for any $x\in Z(R)$. Thus $R$ is reduced. 
\end{proof}
Some times following lemma can be used to find spectra of graphs.
\begin{lemma}\label{l1}
	Let $F$ be a field and  $A.B, D\in M_n(F)$. If $B,D$ are diagonal matrices and $A$ is a symmetric matrix with $AB=BA$ then   $\sigma(B+DAD)=\sigma(B)+\sigma(DAD)$. 
 \end{lemma}
\begin{proof}
Since $A$ is symmetric and $D$ is a diagonal matrix, $DAD$ is a symmetric matrix. There is matrix $P$ such that $P^tP=I$ and $P^tDADP=\Lambda$, where $\Lambda$ is a diagonal matrix and its diagonal entries are eigenvalues  of $DAD$. 
Since $B$ is a diagonal matrix, it is also diagonalizable. If $AB=BA$ then $(DAD)B=B(DAD)$  , which gives $A$ and $B$ are simultaneously orthogonally diagonalizable.
That is, there exist orthogonal matrix $P$ such that each column of $P$ is an eigenvector of $DAD$ as well as $B$. 
Therefore
$P^{t}(B+DAD)P=P^{t}BP+ P^{t}DADP.$
Hence $\sigma(B+DAD)=\sigma(B)+\sigma(DAD).$
\end{proof}
\begin{remark}
	Let each $i=1,2...,n$, $G_i$ is $r_i$ regular graph and $|G_i|=n_i$.\\ Let $\displaystyle G=\bigvee_{H}\{G_1,G_2, \ldots, G_n\}$ and each $G_i$ is $r_i$-regular graph with $|G_i|=n_i$.\\
	Let $B=diag(r_1,r_2, \ldots, r_n),~C=diag(N_1,N_2, \ldots, N_n)~and~D=diag(\sqrt{n_1},\sqrt{n_2}, \ldots, \sqrt{n_n})$. Then 
	$$C_A(H)=B+DA(H)D ~~\text{ and }~~C_N(H)=C+DA(H)D.$$ 
	If $BA(H)=A(H)B$ and $CA(H)=A(H)C$ then by Lemma \ref{l1},
	$$\sigma(C_A(H))=\sigma(B)+\sigma(DA(H)D)~~\text{ and } ~~\sigma(C_N(H))=\sigma(C)+\sigma(DA(H)D).$$ 
\end{remark}

 Now we state the results by Cardoso et al. from \cite{2}. 
\begin{proposition}\label{p1}
	Let $H$ be a graph on set $I=\left\{1,2,\ldots,n\right\}$
	and let $\displaystyle\mathcal{F}=\{ G_i \colon i\in I \}$ be a family of $n$ pairwise disjoint $r_i-$ regular graphs of order $n_i$ respectively. Let $\displaystyle G=\bigvee_{H}\mathcal{F}~~~~~~~~\text{and}~~~~~~~ N_i=\begin{cases}\displaystyle \sum_{j\in N(i)}n_j,& N(i)\neq \phi\\
		0,& \text{otherwise}
	\end{cases}.$\\
	If
	\begin{center}
		$\displaystyle C_A(H)=(c_{ij})=\begin{cases} r_i,& i=j\\
		\sqrt{n_in_j},& i  \text{ adjacent to }j\\
		0,~~& \text{otherwise}

	\end{cases}$
	and\\ $\displaystyle C_N(H)=(d_{ij})=\begin{cases} N_i,& i=j\\
		-\sqrt{n_in_j},& i~~ \text{adjacent to } j\\
		0,& \text{otherwise}

	\end{cases}
	$.
\end{center} then
	\begin{equation}\sigma_{A}(G)=\left(\bigcup_{i}^{n}\left(\sigma_A(G_i)\setminus\left\{r_i\right\} \right)\right) \bigcup\sigma(C_A(H)).
\end{equation}
and
	\begin{equation}		\sigma_{L}(G)=\left(\bigcup_{i}^{n}\left(N_i+(\sigma_L(G_i)\setminus\left\{0\right\})\right)\right) \bigcup\sigma(C_N(H)).
	 \end{equation}
\end{proposition}
\begin{remark}
	Note that in the above proposition, each $G_i$ is $r_i$-regular graph, hence  $[\underbrace{1,1, \ldots, 1}_{n_i~~times}]^t$ is its Perron vector, {\it i.e.}, eigenvector associated to largest eigenvalue $r_i$.
\end{remark}
  
\begin{corollary}\label{c1}
	Let $H$ be a graph on vertices $\{ 1,2, \ldots, t\}$; and\\
	  $G=\displaystyle \bigvee_H\{K_{n_1}, \ldots, K_{n_r},\overline{K}_{n_{r+1}}, \ldots, \overline{K}_{n_{t}}\}$. 
	Then 
	\begin{align}\label{e1}& \sigma_A(G)=\left(\bigcup_{i=1}^{r}\{(-1)^{(n_i-1)} \}\right) \bigcup \left(\bigcup_{i=r+1}^{t}\{0^{(n_i-1)}\}\right)\bigcup \sigma(C_A(H)),~~ \\
		& \sigma_L(G)=\left(\bigcup_{i=1}^{r}\{(N_i+n_i)^{(n_i-1)} \}\right) \bigcup \left(\bigcup_{i=r+1}^{t}\{0^{(n_i-1)}\}\right)\bigcup \sigma(C_N(H)).\nonumber
	\end{align} 
\end{corollary}
\begin{proof}
We have, $\sigma(K_{n_i})=\{ (n_i-1)^{(1)}, (-1 )^{n_i-1}\}$
for each $i=1,2...,r$ and $\sigma(\overline{K_{n_i}})=\{ 0^{n_i}\}$ for each $i=r+1,...,t$.
Expressions for $\sigma_A(H)$ and $\sigma_L(H)$ in  (\ref{e1}) are evident from Proposition \ref{p1}. 
\end{proof} 

If $R$ is a finite ring with unity, then the adjacency matrix $A(\Gamma(R))$ is obtained from $A(\Gamma(R)^{\sim})$ as below. 
For a finite ring with unity, we write $\sigma_A(\Gamma(R))$ and $\sigma_L(\Gamma(R))$ using the generalized join operation.
 \begin{proposition}
 Let $R$ be a finite ring with unity and $\Gamma(R)^{\sim}=\{[x_1],[x_2],...,[x_r],[x_{r+1}], \ldots, [x_t]\}$ with  $x_i^2=0$ for $i=1,2,...,r$.    Suppose that $n_i=|[x_i]|$, for $i=1,2,\ldots, t$.
 Then  
 \begin{enumerate} 
 \item 
 $\displaystyle\Gamma(R)=\bigvee_{\Gamma(R)^{\sim}}\{K_{n_1}, \ldots, K_{n_r},\overline{K}_{n_{r+1}}, \ldots, \overline{K}_{n_{t}}\}$. 
 \item 
  $\displaystyle\sigma_A(\Gamma(R))=\left(\bigcup_{i=1}^{r}\{(-1)^{(n_i-1)} \}\right) \bigcup \left(\bigcup_{i=r+1}^{t}\{0^{(n_i-1)}\}\right)\bigcup \sigma(C_A(G^\sim))$.
\item $\displaystyle \sigma_L(\Gamma(R))=\left(\bigcup_{i=1}^{r}\{(N_i+n_i)^{(n_i-1)} \}\right) \bigcup \left(\bigcup_{i=r+1}^{t}\{0^{(n_i-1)}\}\right)\bigcup \sigma(C_N(\Gamma(R)^\sim))$, \\
where $\sigma(C_A(G^\sim))$ and  $\sigma(C_N(\Gamma(R)^\sim))$ are as given in Corollary \ref{c1} with $i$ replaced by $x_i$. Also $r_i=n_i-1$, for $i=1,2,\ldots,r$ and $r_i=0$, for $i=r+1,\ldots,t$.
\end{enumerate}
 \end{proposition} 
\begin{proof}
	Follows from Propositions \ref{p2}, \ref{p1} and Corollary \ref{c1}.
	\end{proof}
Let $n$ be a positive integer and $V=\{i \in \mathbb{N} \colon 1<i<n, i~\text{divides }~n \}$.  Chattopadhyay et al. \cite{ch} defined the simple graph $\Upsilon_n$ whose vertex set is $V$ in which two distinct vertices $i$ and $j$ are adjacent if and only if $n$ divides $ij$. They have shown that $\displaystyle \Gamma(\mathbb{Z}_n)=\bigvee_{\Upsilon_n}\Gamma(A_i)$, where where $A_i=\{x\in \mathbb{Z}_n \colon (x,n)=i\}$. Observe that, $A_i=[i]^{\sim}$, for each $i$ and $\Upsilon_n=\Gamma(\mathbb{Z}_n)^\sim$. Thus we have essentially extended the results of Chattopadhyay et al. \cite{ch} to finite rings with unity. 
In the following result, we prove that any graph $G$ is a $G^{\approx}$-generalized join of its induced subgraphs on equivalence classes of the relation $\approx$.
 Let $G$ be any graph and $G^{\approx}$ be its compressed graph. For each vertex $x\in G,$ $[x]^{\approx}$ denotes the equivalence class of $\approx$ containing $x$.  Also, $G_{[x]^{\approx}}$ is an induced subgraph of $G$ on $[x]^{\approx}$.
\begin{proposition}\label{p3}
Let $G$ be a graph and for each vertex $x\in G$, $G_{[x]^{\approx}}$ be an induced subgraph of the graph $G$ on  $[x]^{\approx}$.
If $|G_{[x]^{\approx}}|=n_x$ then   $G_{[x]^{\approx}}=\overline{K_{n_{x}}}$ and 
$G$ is $G^{\approx}-$ generalized join of  null graphs.
\end{proposition}
\begin{proof}
Let $[x]^{\approx}, [y]^{\approx}\in G^{\approx}$ and  $u\in [x]^{\approx},~~v\in[y]^{\approx}. $  If $[x]^{\approx}=[y]^{\approx}$ then $N(x)=N(y)=N(u)=N(v)$. Therefore $x-y$ and $u-v$ are not edges in the graph $G$.
Suppose that $[x]^{\approx}\neq [y]^{\approx}$. Hence $[x]\cap [y]=\emptyset$.
Suppose $x-y$ is an edge. Therefore $y\in N(x)=N(u)$. If $u-v$ is not an edge in the graph G then  $v\notin N(u)=N(x)$.  Hence $x\notin N(v)=N(y)$. This gives $x-y$ is not an edge in $G$, which is a contradiction. Therefore, if $x-y$ is an edge in $G$ then $u-v$ is an edge in $G$. Similarly if $u-v$ is an edge then $x-y$ is also an edge. Therefore $x-y$ is edge if and only if $u-v$ is an edge. Hence $G$ is $G^{\approx}-$ generalized join of induced subgraphs  on distinct equivalence classes of $\approx$.

Now we will show that each graph  $G_{[x]^{\approx}}$ is a null graph. Let $u,v\in [x]^{\approx}$. Then $N(u)=N(v)=N(x)$. If $u-v$ is an edge, then $N(u)\neq N(v),$ a contradiction to  $N(u)=N(v)=N(x)$. Therefore each $G_{[x]^{\approx}}$ is a null graph. 
\end{proof}
 \begin{corollary}\label{c2}
 Let $R$ be a finite ring with unity and $\{[x_i]^{\approx}~|~ i=1,2,3...,m\}$ be distinct equivalence classes of $\approx$ on $Z(R)$. Suppose that $n_i=|[x_i]^{\approx}|$, for $i=1,2,\ldots, m$. 
 Then   $\displaystyle \Gamma(R)=\bigvee_{G^{\approx}}\{\overline{K_{n_1}}, \ldots, \overline{K}_{n_{m}} \},$ and
  $$\sigma_A(\Gamma(R))= \left(\bigcup_{i=1}^{m}\{0^{(n_i-1)}\}\right)\bigcup \sigma(C_A(\Gamma(R)^\approx))),
\sigma_L(\Gamma(R))= \left(\bigcup_{i=1}^{m}\{0^{(n_i-1)}\}\right)\bigcup \sigma(C_N(\Gamma(R)^\approx)),$$ 
where $C_A(\Gamma(R)^\approx))$ and  $C_N(\Gamma(R)^\approx))$ are as given in Corollary  \ref{c1} with $i$ replaced by $x_i$. 
 \end{corollary}
\begin{proof}
The proof follows from Proposition \ref{p3}.
\end{proof}
A ring $R$ is \textit{regular} (\textit{von-Neumann regular}) if for any  $a\in R$, there exist $b\in R$ such that $a=aba$. Let $a=aba$ for some $a, b\in  R$ and $e=ab$, $ba=f$. Observe that, $e^2=e$ and $f^2=f$ and  $ann_r(a)=ann_r(e)=(1-e)R$ and $ann_l(a)=ann_l(f)=R(1-f)$. A ring $R$ is said to rickart if for any $a\in R$ there exist idempotent $e$ such that $ann_r(a)=eR$ or $ann_l(a)=Re$. Therefore regular rings are rickart rings.
\begin{proposition}(Beiranvand  et al. \cite{8}, Proposition 2.4).
Every finite commutative regular ring or finite reduced Goldie ring is finite direct product of finite fields
\end{proposition} 
\begin{proposition}( Thakare et al. \cite[Theorem 6]{9})
A $*$ ring with finitely many elements is Bear $*$ ring if and only if $A=A_1\oplus A_2\oplus...\oplus A_r$ where $A_i$ is a field or $A_i$ is a $2\times 2$ matrix ring over finite field $F(p^n)$ with $n$ odd positive integer and $p$ is a prime  of the form $4k+3$.
\end{proposition}
 
 \begin{proposition}
 If ring $R$ is  finite commutative Rickart ring or finite Von -Neumann regular commutative ring then it is  finite direct product of finite fields.
 \end{proposition}
 \begin{proposition}
 ( Patil et al. \cite{ap}) Let $R$ be finite commutative Von Neumann regular ring with set of nontrivial idempotents $B(R)=\{ e_i~|~i=1,2,...,r\}$, $A_{e_i}=\{ x\in R ~|~ ann(x)=ann(e_i)\}$  for $i=1,2,...r$, $\Gamma(B(R))$ is induced subgraph of $\Gamma(R)$ on $B(R)$ and $c_{ij}=\sqrt{|A_{e_i}||A_{e_j}|}$ . Then 
 $$\sigma_A(\Gamma(R))=\{ 0^{(|Z(R)*|-r)} \}\cup \sigma(C)\}$$, where $C$ is matrix whose $(i,j)^{th}$ entry is zero if $e_ie_j\neq 0$ and $c_{ij}$ if $e_ie_j=0$.
 and 
 $$\sigma_L(\Gamma(R))=\{ M_{e_i}^{(|A_{e_i}|-1)} \}\cup \sigma(\sigma_L(\Gamma(B(R))))\},$$
 where $\displaystyle M_{e_i}=\sum_{j,~ e_je_i=0}|A_{e_j}|$ for $i=1,2,...,r$.
 \end{proposition}
  \begin{proposition}
 (In John D. Lagrange,)
 Let $R$ be a ring. Then $R$ is a Boolean ring if and only if the set of eigenvalues $\sigma(\Gamma(R))$ (counting with multiplicities) is partitioned into 2-element subsets of form $\{\lambda, \pm \frac{1}{\lambda}\}$
 \end{proposition} 
Let  $R$ is a direct product of finite number of finite fields. In the following lemma, we expressed the zero-divisor graph $\Gamma(R)$ as a generalized join graph. Further, we compute adjacency and Laplacian spectra of $\Gamma(R)$ in terms of spectra of the Boolean ring.\\
Let $q_k=p_k^{m_k}$ with $p_k$ prime and $F_{q_k}$ be finite field, for $k=1,2,\ldots,t$;
and $R=F_{q_1}\times F_{q_2}\times\cdots \times F_{q_t} $ be a ring. Let 
\begin{align*}
&\displaystyle e_1=(1,0, \ldots, 0),~e_2=(0,1,0, \ldots, 0), \ldots, e_t=(0,0, \ldots, 1),\\
  &\displaystyle \mathcal{A}_1=\left\{ e_1,e_2,...,e_t\right\},\\
  &\displaystyle \mathcal{A}_2=\left\{ e_1+e_2,~ e_1+e_3,...,e_1+e_t, e_2+e_3,...,e_2+e_t,...,e_{t-1}+e_t\right\},\\
  &\displaystyle \mathcal{A}_3=\left\{e_1+e_2+e_3,....,e_{t-2}+e_{t-1}+e_t\right\},....\\
  &\displaystyle \mathcal{A}_{t-1}=\left\{e_1+e_2+...+e_{t-1},~e_2+e_3+...+e_{t}\right\}.
  \end{align*}
  be orderd sets.
  Then $\mathcal{A}=\mathcal{A}_1\cup \mathcal{A}_2\cup...\cup\mathcal A_{t-1}$ is an ordered set of all  idempotents  $Z(R)$. 
  For any $e\in \mathcal{A}$, let $S_e=\{ i~|~ e.e_i\neq 0\}$.
\begin{lemma}
Let $q_k=p_k^{m_k}$ with $p_k$ prime and $F_{q_k}$ be finite field, for $k=1,2,\ldots,t$;
and $R=F_{q_1}\times F_{q_2}\times\cdots \times F_{q_t} $ be a ring,   $\mathcal{F}=\left\{ \Gamma([e_{S}]^{\sim}) \colon e_S\in \mathcal{A} \right\}$ and $\displaystyle U(R)$ is the set of units in $R$.
 Then
\begin{enumerate}
\item $\displaystyle \Gamma(R)^{\sim}=\Gamma(\mathcal{A})=\Gamma\left(\bigoplus_{i=1}^{t}\mathbb{Z}_2 \right)$ and  $\displaystyle \Gamma(R)=\bigvee_{\Gamma(R)^{\sim}}\mathcal{F}$.
\item $|\Gamma(R)^{\sim}|=|\mathcal{A}|=2^t-2$.
\item  $\displaystyle |[e]^{\sim}|=\prod_{i\in S_e}(q_i-1)=n_e~(say)$,
\item $\displaystyle N([e]^{\sim})=(1-e)\mathcal{A}\setminus\{0\},~$ $d([e]^{\sim})=2^{t-|S_e|}-1$.
\item $\displaystyle N(e)=(1-e)R\setminus\{0\},~~d(e)=\prod_{i\notin S_e}q_i-1$.
\item $
\displaystyle \sigma_A(\Gamma(R))=\left(\bigcup_{e\in \mathcal{A}}\{0^{(n_e-1)}\}\right)\bigcup \sigma\left(B+D A\left(\Gamma\left(\bigoplus_{i=1}^{t}\mathbb{Z}_2\right)\right)D\right), \\ \displaystyle \sigma_L(\Gamma(R))=\left(\bigcup_{e\in \mathcal{A}}\{0^{(n_e-1)}\}\right)\bigcup \sigma\left(C+DA\left(\Gamma\left(\bigoplus_{i=1}^{t}\mathbb{Z}_2\right)\right)D\right),$\\
Where $B=[d(e)]_{e\in \mathcal{A}}$  and $C=[n_e]_{e\in\mathcal{A}}$ are diagonal matrices.
\end{enumerate}
 \end{lemma}
\begin{proof}
Let $a=(a_1,a_2,...a_t)\in Z(R)$ and  $e_{a_i}=1$ if $a_i\neq 0$ and $e_{a_i}=0$ if $a_i=0$ for each $i=1,2,...,t$. Hence $e_a=(e_{a_1},...e_{a_t})$ is an idempotent in $R$ uniquely determined by $a$ such that $e_a\sim a$. Clearly set of all idempotent in $R$ forms Boolean ring $\displaystyle \bigoplus_{i=1}^{t}\mathbb{Z}_2$. 
For $a, b\in Z(R)$, $a-b$ is an edge if and only if $e_a-e_b$ is an edge. 
 Therefore (1) holds true.\\
Now $\displaystyle |\Gamma(R)^{\sim}|=|\mathcal{A}|=\sum_{i=1}^{t-1}|\mathcal{A}_i|=2^t-2$. Hence (2) is true.\\
Clearly, $\displaystyle [e]^{\sim}=\{x\in Z(R) ~|~ x=eu,~\text{for some u in~} U(R)\}=eU(R),~~$. Hence (3) is true.\\
Let $f$ be a vertex in $\Gamma(R)^{\sim}$.   $f\in N([e]^{\sim})$ if and only if $ef=0$ ie., $f=(1-e)\mathcal{A}$ and $f\neq 0$. Hence (4) is true.\\
Let $x$ be a vertex in $\Gamma(R)$.   $f\in N([e]^{\sim})$ if and only if $ex=0$ ie., $x=(1-e)R$ and $x\neq 0$. Hence (5) is true.\\
Proof of (6), follows from proposition (\ref{p2}).
\end{proof}
\begin{proposition}
Let $R$ be a finite, abelian and  regular ring with unity. If $\mathcal{I}$ denotes the set of all idempotent in $R$, then the following statements hold: 
\begin{enumerate}
\item $\mathcal{F}=\{\Gamma([e]^{\sim}) \colon e\in \mathcal{I} \}$ is family of null graphs.
\item $\Gamma(R)^{\sim}=\Gamma(\mathcal{I})$.
 \item $\Gamma(R)=\bigvee_{\Gamma(R)^\sim} \mathcal{F}$.
 \end{enumerate}
\end{proposition}
\begin{proof} Since abelian regular rings are reduced, $R$ is a reduced ring. Hence for any $r\in R$ there exist an idempotent $e$ and a unit $u$ such that $r=ue=eu$. (see Beiranvand et al.  \cite[Remark 3.4]{8}. If there is another idempotent $f$ and unit $v$ such that $r=vf=fv$, then we have $ue=eu=fv=vf$. Consequently, $(1-f)eu=0=v(f-fe)$. Therefore $e=f=ef$. Hence for any $r\in R$ there exist a unique idempotent, say $e_r$, such that $r\sim e_r$. Hence for any $r\in Z(R),$ there exist a unique idempotent $e_r$ such that $e_r\in [r]^{\sim}$. Hence by Proposition \ref{p2}, $\Gamma(R)$ is $\Gamma(R)^{\sim}$ generalized join of graphs in the family $\mathcal{F}$. Since $R$ is a reduced ring, from  Proposition \ref{p2}, each graph $\Gamma([e]^{\sim})$ is a null graph.  This proves statements $(1),(2)$ and $(3)$. 
\end{proof}

\section{spectra of the zero-divisor graph of $\mathbb{Z}_n$ and $M_n(F_q)$} 

 Recall the following remarks, which are useful in this section.
\begin{remark} [{\cite{1}}]
	Denote the complete graph  of order $n$ and its complement, {\it i.e.}, the null graph of order $n$, by $K_n$ and $\overline{K_n}$ respectively. Since $A(\overline{K_n})=L(\overline{K_n})$   is a zero matrix of order n,  $\sigma_A(\overline{K_n})=\sigma_L(\overline{K_n})=\left\{ 0^{(n)}\right\}$. \\
	Note that $A(K_n)=J_n-I_n$, where $J_n$ is a matrix of order $n$ of all 1's and $I_n$ is the identity matrix of order $n$. Therefore   $\sigma_A(K_n)=\left\{(-1)^{(n-1)}, (n-1)^{(1)} \right\}$.
	Also, $L(K_n)=(n-1)I_n-A(K_n)=nI_n-J_n$. Hence   $\sigma_L(K_n)=\left\{n^{(n-1)}, 0^{(1)}\right\}$.
\end{remark}
\begin{remark} [{\cite{6}}]
	Let $q=p^k$ with $p$ prime. Then
	$ \displaystyle {n\choose r}_q=\frac{\prod_{i=0}^{r-1}(q^n-q^{i})}{\prod_{i=0}^{r-1}(q^r-q^{i})}$
	is called as $q$-\textit{binomial coefficient}.
	
	The following properties of $q$- binomial coefficients are used in the sequel.
	\begin{enumerate}
		\item
		$\displaystyle{n\choose r}_q=0~~~\text{, if }~~r>n~~\text{ or }~~r<0$.
		\item 
		$\displaystyle{n\choose r}_q={n\choose n-r}_q$.
		\item 
		$\displaystyle{n\choose 0}_q={n\choose n-1}_q=1$.
		\item  $\displaystyle{n\choose 1}_q={n\choose n-1}_q=\frac{q^n-1}{q-1}~~\text{, if }~~n\geq 1$.
		\item $\displaystyle \lim_{q\longrightarrow 1}{n\choose r}_q={n\choose r}$.
		\item $\displaystyle \sum_{r=0}^{n}q^{r^2}{n\choose r}_q={2n\choose n}_q$.
		\item The number of linearly independent subsets of cardinality $r$ of $n$-dimensional vector space over a finite field $F_q$ is
		$\displaystyle \frac{\prod_{i=0}^{r-1}(q^n-q^i)}{r!}$.
		\item The number of $r$ dimensional subspaces of $n$-dimensional vector space over a finite field $F_q$ is 
		$\displaystyle{n\choose r}_q$.
	\end{enumerate}
\end{remark} 
In \cite{7}, Khaled et al. listed the following result which gives the number of matrices of given rank and given size over a finite field. This proposition is useful in determining cardinality of some sets. 
\begin{proposition}\label{p7}
	The number of matrices of size $n\times m$ of rank $r$ over finite field of order $q$ is 
	$$M(n,m,r,q)=\prod_{j=0}^{r-1}\frac{(q^n-q^j)(q^m-q^j)}{(q^r-q^j)}.$$
\end{proposition}
In \cite{ch} Chattopadhyay et al. gave the adjacency and Laplacian spectra of $\Gamma(\mathbb{Z}_n)$. In this section, we determine the adjacency and Laplacian spectra of $\Gamma(\mathbb{Z}_n)$ and $\Gamma(M_n(F_q))$ using the results proved in previous sections.

\subsection{Spectra of $\mathbf{\Gamma(\mathbb{Z}_n)}$}.

 The following theorem will be used to find the Spectra of $\Gamma(\mathbb{Z}_n)$.
\begin{theorem}
Let $R=\mathbb{Z}_n$.  If $d_1, d_2, \ldots, d_k$ are nontrivial divisors of $n$, then $\displaystyle \Gamma(\mathbb{Z}_n)=\bigvee_{\Gamma(R)^{\sim}}\{\Gamma([d_1]), \ldots, \Gamma([d_k])\}$.
 And each  $\Gamma([d_i])$ is either a complete graph or a null graph. Moreover, $\Gamma([d_i])$ is a complete graph if and only if $n$ divides $d_i^2$.
 \end{theorem}
\begin{proof}
Proof follows from Proposition \ref{p5} and Proposition \ref{p2}.
\end{proof}
Let $R=\mathbb{Z}_n$ be a ring and
\begin{align*}
&N_2=\{[d_i] \colon d_i\neq 0, d_i^2=0~~\text{ in}~~ \mathbb{Z}_n\}~~\text{be a set of nonzero nilpotents of index 2},\\
& L=Z(R)\setminus N_2.
\end{align*}
\begin{proposition}\label{p6}
Let $\displaystyle n=\prod_{i=1}^{t}p_i^{k_i}$  and $R=\mathbb{Z}_n$ be a ring. Then $$\displaystyle |N_2|=\prod_{i=1}^{t}\left[\frac{k_i}{2}\right]-1=s~say$$ and $$|L|=\displaystyle \prod_{i=1}^{t}(k_i+1)-1-s=l ~say.$$  Also following statements hold.
\begin{enumerate}
\item  $\displaystyle|\Gamma(R)^{\sim}|=\prod_{i=1}^{t}(k_i+1)-2$; and for any two divisors $\displaystyle d=\prod_{i=1}^{t}p_i^{\alpha_i},~~~d'=\prod_{i=1}^{t}p_i^{\beta_i}$, the vertices $[d], [d']$ are adjacent  in $\Gamma(R)^{\sim}$ if and only if $k_i\leq \alpha_i+\beta_i$, for all $i=1,2,\ldots,t$.
\item For each divisor $\displaystyle d=\prod_{i=1}^{t}p_i^{\alpha_i}$ of n, $\displaystyle |[d]|=\prod_{i=1}^{t}\left(p_i^{k_i-\alpha_i}-p_i^{k_i-\alpha_i-1}\right)=n_d~(say)$.
\item For each divisor $\displaystyle d=\prod_{i=1}^{t}p_i^{\alpha_i}$ of n, vertex $[d]$  in $\Gamma(R)^{\sim}$ has a degree  $\displaystyle \prod_{i=1}^{t}\left(\alpha_i+1\right)$.
\item  For each divisor $\displaystyle d=\prod_{i=1}^{t}p_i^{\alpha_j}$ of n, the vertex $d$  in $\Gamma(\mathbb{Z}_n)$  has a degree\\
$\displaystyle \sum_{k_i-\alpha_i\leq \beta_i\leq k_i }\prod_{i=1}^{t}\left(p_i^{k_i-\beta_i}-p_i^{k_i-\beta_i-1}\right)$.
\end{enumerate} 
\end{proposition}
\begin{proof}
Let $\displaystyle n=\prod_{i=1}^{t}p_i^{k_i}$. The number of  nontrivial divisors of $n$ is equal to $l=\displaystyle \prod_{i=1}^{t}(k_i+1)-2$ and number of units in $\mathbb{Z}_n$ is $\displaystyle\phi(n)=\prod_{i=1}^{t}(p_i^{k_i}-p_i^{k_i-1})$.\\ Let $d=\prod_{i=1}^{t}p_i^{\alpha_i}$ be a divisor of $n$.  We count the number of associates of $d$ in $\mathbb{Z}_n$. Let $d_i=p_i^{\alpha_i}$. Now $\displaystyle d\longrightarrow (d_1,d_2,\ldots,d_t)$ is a bijective map from $\mathbb{Z}_n$ to $\displaystyle \mathbb{Z}_{p_1^{k_1}}\times \ldots\times \mathbb{Z}_{p_t^{k_t}}$. Two elements $\displaystyle d=\prod_{i=1}^{t}p_i^{\alpha_i},~~d'=\prod_{i=1}^{t}p_i^{\beta_i}$ are associates in $\mathbb{Z}_n$ if and only if $\displaystyle d_i=p_i^{\alpha_i}$ and $\displaystyle d_i'=p_i^{\beta_i}$ are associates in $\mathbb{Z}_{p_i^{k_i}}$, for all $i=1,2,\ldots, t$. Hence the number of associates of $d$ is equal to $\displaystyle \prod_{i=1}^{t}n_i,$ where $\displaystyle n_i=\text{number of associates of}~ p_i^{\alpha_i}~ in~\mathbb{Z}_{p_i}^{k_i}$. The set of associates of $\displaystyle p_i^{\alpha_i}$ in $\displaystyle\mathbb{Z}_{{p_i}^{k_i}}$ is $$\left\{rp_i^{\alpha_i}~\colon ~(r,p_i^{k_i})=1,~1\leq rp_i^{\alpha_i}<p_i^{k_i}  \right\}= \left\{rp_i^{\alpha_i}~\colon ~(r,p_i^{k_i})=1,~1\leq r<p_i^{k_i-\alpha_i}  \right\}$$
$$=\left\{rp_i^{\alpha_i}~\colon ~r\in \{1,2,\ldots,p_i^{k_i-\alpha_i}\} ~\text{and}~ r\neq p_is,~ \text{for some } s\in \mathbb{Z}_{p_i^{k_i}} \right\}.$$
Hence the number of associates of $\displaystyle p_{i}^{\alpha_i}$ is $\displaystyle (p_i^{k_i-\alpha_i}-p_i^{k_i-\alpha_i-1})$.
Therefore the number of associates of $d$ is  $\displaystyle \prod_{i=1}^{t}\left( p_i^{k_i-\alpha_i}-p_i^{k_i-\alpha_i-1}\right)$.
Hence for each divisor $\displaystyle d=\prod_{i=1}^{t}p_i^{\alpha_i}$ of $n,$  $\Gamma([d])$ is a graph on $\displaystyle n_d=\prod_{i=1}^{t}\left( p_i^{k_i-\alpha_i}-p_i^{k_i-\alpha_i-1}\right)$ vertices. Alternatively $n_d=\phi\left( \frac{n}{d}\right),$ because associates of $d$ lie in a cyclic subgroup of $\mathbb{Z}_n$ generated by $d$.\\ Now we count the degree of each vertex $d$ in $\Gamma(\mathbb{Z}_n)$. If $\displaystyle d'=\prod_{i=1}^{t}p_{i}^{\beta_i}$ is such that $dd'=0$ in $\mathbb{Z}_n$, then $n$ divides $dd'$. Hence we have $k_i\leq \alpha_i+\beta_i$, for each $i$. Therefore $k_i-\alpha_i\leq \beta_i\leq k_i$, for each $i=1,2,\ldots,t$. Thus the number of neighbors of $[d]$ in $\Gamma(R)^{\sim}$ is $\displaystyle \prod_{i=1}^{t}(\alpha_i+1)$. Also, the number of neighbors of $d$ in $\Gamma(\mathbb{Z}_n)$ is $\displaystyle\sum_{n|dd',~d'|n}|[d']| =\sum_{n|dd',~d'|n}\phi(d')=\sum_{k_i-\alpha_i\leq \beta_i\leq k_i }\prod_{i=1}^{t}\left(p_i^{k_i-\beta_i}-p_i^{k_i-\beta_i-1}\right)$.
\\ Let $s$ be the number of vertices $[d]$ in $\Gamma(R)^{\sim}$ such that $d$ is a nilpotent element of index two. Hence $s$ is the number of complete subgraphs of type $\Gamma([d])$ in $\Gamma(\mathbb{Z}_n)$. Each nonzero nilpotent divisor $d$  of n having index two in $\mathbb{Z}_n$ is of the form $\displaystyle \prod_{i=1}^{t}p^{m_i}$ with $\displaystyle k_i\leq 2m_i~~\text{and}~~\sum_{i=1}^{t}m_i< \sum_{i=1}^{t}k_i$. Hence $s=\displaystyle \prod_{i=1}^{t}\left[\frac{k_i}{2}\right]-1$.
Therefore  $\Gamma(R)^{\sim}$ is a graph with vertex set $V=\left\{ [d] \colon d ~\text{is nontrivial divisor of}~n  \right\}$ and edge set $E=\left\{ \left\{[d],[d']\right\} \colon d, d'\in~ V~\text{and}~n~\text{divides}~dd'\right\}$.
Also,  $\Gamma(\mathbb{Z}_n)$ is $\Gamma(R)^{\sim}-$ a generalized join of family of graphs  $\left\{ \Gamma([d]) \colon d ~~\text{is divisor of}~~ n \right\}$ with $s$ complete graphs and $l-s$ null graphs.
\end{proof}
Finally, we give spectra of the zero-divisor graph of $\mathbb{Z}_n$.
\begin{theorem}
 Let $\displaystyle n=\prod_{i=1}^{t}p_i^{k_i}, l=\prod_{i=1}^{t}(k_i+1)-2, s=\prod_{i=1}^{t}\left[\frac{k_i}{2}\right] $.  Let $N$ be the set of all nontrivial divisors of $n$, $N_2$ be the set of divisors of $n$ having nilpotency index two and for each divisor $\displaystyle d_i=\prod_{j=1}^{t}p_j^{\alpha_j},~~$ $\displaystyle n_{d_i}=\prod_{j=1}^{t}\left(p_j^{k_j-\alpha_j}-p_j^{k_j-\alpha_j-1}\right),$    $\displaystyle N_{d_i}=\sum_{n|d_id_j}n_{d_j}$.\\
 Then $\Gamma(\mathbb{Z}_n)$ is  $\Gamma(\mathbb{Z}_n)^{\sim}-$ generalized join of graphs\\ $\displaystyle \left\{\Gamma([d_i])\simeq \overline{K_{n_{d_i}}} \colon d_i\in N\setminus N_2\right\}\bigcup \left\{\Gamma([d_i])\simeq K_{n_{d_i}} \colon d_i\in N_2\right\}$. Also,
 \begin{enumerate} 
 \item  
 $\displaystyle \sigma_A(\Gamma(\mathbb{Z}_n))=\left(\bigcup_{d_i\in N_2}\left\{-1^{(n_{d_i}-1)}\right\}\right)\bigcup\left( \bigcup_{d_i\in N\setminus N_2}\left\{ 0^{(n_{d_i}-1)} \right\}\right) \bigcup \sigma(C_A(\Gamma(\mathbb{Z}_n)))$,\\
 where $C_A(\Gamma(\mathbb{Z}_n))$ is a square matrix of order $l$ defined as below. 
 If $N_2=\left\{d_1,d_2, \ldots, d_s\right\}$ and $N\setminus N_2=\left\{d_{s+1}, \ldots, d_l \right\}$, then 
 $$C_A(\Gamma(\mathbb{Z}_n))=(c_{ij})=\begin{cases} n_{d_i}-1 ~~~& d_i=d_j ~\text{and}~ d_i\in N_2\\
                                 \sqrt{n_{d_i}n_{d_j}}~~& d_i~~ \text{adjacent to}~~ d_j\\
                                 0~~& \text{otherwise}

          \end{cases}
 .$$
 \item $\sigma_L(\Gamma(\mathbb{Z}_n))=\left(\bigcup_{d_i\in N_2}\left\{(N_{d_i}+n_{d_i})^{(n_{d_i}-1)}\right\} \right) \bigcup \left( \bigcup_{d_i\in N \setminus N_2}\left\{N_{d_i}^{(n_{d_i}-1)}\right\}\right)\bigcup \sigma(C_N(\Gamma(\mathbb{Z}_n)))$,\\
 where 
 $$C_N(\Gamma(\mathbb{Z}_n))=(c_{ij})=\begin{cases} N_{d_i} ~~~& d_i=d_j\\
                                  -\sqrt{n_{d_i}n_{d_j}}~~& d_i~~ \text{adjacent to}~~ d_j\\
                                  0~~& \text{otherwise}

           \end{cases}
  .$$
 \end{enumerate}
 \end{theorem}
 \begin{proof} 
 The proof is clear from Propositions \ref{p2}, \ref{p1}, and \ref{p6}.
 \end{proof}

 \subsection{Spectra of $\mathbf{\Gamma(M_n(F_q))}$}
 
Let $p$ be a prime, $~q=p^k$  and $M_n(F_q)$ be a matrix ring of $n\times n$ matrices over a finite field $F_q$. 
The following lemma gives the cardinality  of every equivalence class  of the relation $\sim$ on $Z(M_n(F_q))$.  
\begin{lemma}\label{l2}Let $A\in Z( M_n(F_q))$.
If $rank(A)=r$, then 
 $\displaystyle |[A]|=\prod_{i=0}^{r-1}(q^r-q^i)$. 
\end{lemma}
\begin{proof}
Let $G^{o}$ be a group with nonempty set $GL_n(F_q)$ together with the binary operation $(U, V)\longrightarrow V.U$. Let $G=G^o\times G^o$ be the external direct product of groups and  $X=M_n(F_q)$.  Consider the map $f:G\times X \longrightarrow X$ defined by $f((P,Q),A)=PAQ^{-1}$, for all $A\in X$ and $(P,Q)\in G$.
This map is an action of group $G$ on $X$. Therefore $$|O(A)|=[G:S_A]=\frac{|G|}{|S_A|},$$ where 
$O(A)=\left\{PAQ^{-1} \colon P, Q\in G\right\}$  is the orbit of the action containing $A$ and $S_A=\{ (P,Q)\in G \colon PAQ^{-1}=A \}$ is a stabilizer subgroup of $A$. Let $A\in Z(M_n(F_q))$ and $rank(A)=r$. 
 Now $O(A)$ is a set of all matrices in $M_n(F_q)$ which are equivalent to $A$, which will consist of all matrices of rank $r$ in $M_n(F_q)$. Hence from  Proposition \ref{p7}, $\displaystyle |O(A)|=\frac{\prod_{i=0}^{r-1}(q^n-q^i)^2}{\prod_{i=0}^{r-1}(q^r-q^i)}$. Also, it is known that $\displaystyle |G|=\prod_{i=0}^{n-1}(q^n-q^i)^2$.
Therefore\\ $\displaystyle |S_A|=\frac{|G|}{|O(A)|}=\left(\prod_{i=r}^{n-1}(q^n-q^i)^2 \right) \left(\prod_{i=0}^{r-1}(q^r-q^i)\right)$.\\
Let $T=\left\{(P,Q)\in S_A \colon PA=A=AQ\right\}$.   Let $(P_1, Q_1), (P_2, Q_2)\in S_A$.  Hence $P_1A=AQ_1$ and $P_2A=AQ_2$. If  $(P_1,Q_1),(P_2,Q_2)$ gives same element in $O(A)$ under the group action ie., $P_1A=P_2A,~AQ_1=AQ_2$ then $P_1^{-1}P_2A=A=AQ_2Q_1^{-1}$. Therefore $(P=P_1^{-1}P_2, Q=Q_2Q_1^{-1})$ is in $T$. Thus, if $(P_1,Q_1)$ and $(P_2,Q_2)$ in $S_A$ gives same element in $O(A)$ then $(P_2,Q_2)=(P_1P, QQ_1)$ with $(P,Q)\in T$.\\ Conversely,  If $(P_1, Q_1)\in S_A$ and $(P_2=P_1P,~Q_2=QQ_1)$ with $(P,Q)\in T$ then $P_1A=P_2A$ and $Q_1A=Q_2A$, ie. $(P_1,Q_1)$ and $(P_2,Q_2)$ gives same element in $O(A)$  under the group action. So $|[A]|=\frac{|S_A|}{|T|}$.

Now we will find $|T|$.
Let $\{X_1, \ldots, X_r\}$ be a basis of a column space of $A$. Let $\{X_1, \ldots, X_r,Y_{r+1}, \ldots, Y_n\}$ be a basis of $F_q^n$.  Therefore $(P,Q)\in T$ if and only if 
$PX_1=X_1,\ldots,PX_r=X_r$ and $\{PY_{r+1},\ldots,PY_n\}$ is a  basis of complementary subspace of the column space of $A$. Hence the cardinality of $S$ is equal to the number of choices of $B=\{PY_{r+1},\ldots,PY_n\}$. Note that $PY_{r+i}\notin span\{X_1,X_2, \ldots, X_r, \ldots, PY_{i} \}$. Hence the total number of choices for $B$ is $\prod_{i=r}^{n-1}(q^n-q^i)$. Thus for each choice of $B,$ the matrix $P$ is uniquely determined. Therefore total choices for $P$ are $\prod_{i=r}^{n-1}(q^n-q^i)$. Now $Q^{-1}A=A$ imply that total choices for $Q$ are also same as that of $P$. Therefore $\displaystyle |T|=\left(\prod_{i=r}^{n-1}(q^n-q^i)\right)^2.$ Hence $\displaystyle|[A]|=\prod_{i=0}^{r-1}(q^r-q^i)$.
\end{proof}
Now for any $A\in M_n(F_q),$ we have
$ann_r(A)=\{ B\in Z(M_n(F_q)) \colon AB=0\}$ and 
$ann_l(A)=\{ B\in Z(M_n(F_q)) \colon BA=0\}$.
If $E$ and $F$ are row reduced echelon and column reduced echelon form of $A$ respectively, then there exist invertible matrices $P$ and $Q$ such that $A=PE=FQ$. Therefore we have
$ann_r(A)=ann_r(E)$ and $ann_l(A)=ann_l(F)$. Also note that $E^2=E$ and $F^2=F$. In the following lemma, we find the degree of $A$ in $\Gamma(M_n(F_q))$.
\begin{lemma}\label{l5}
Let $A\in Z( M_n(F_q))$ and $A^2\neq 0$.
If $rank(A)=r$, then 
 $$\displaystyle d(A)=2q^{n(n-r)}-q^{(n-r)^2}-1.$$ 
\end{lemma}
\begin{proof}Let $R=M_n(F_q)$, $A\in R$ and $rank(A)=r$. Degree of $A$ is given by\\ $d(A)=|N(A)|=|ann_r(A)|+|ann_l(A)|-|ann_r(A)\cap ann_l(A)|-1.$ \\Let $E$ be an idempotent obtained from reduced row echelon form of $A$ by interchanging row, so that that leading 1's on the diagonal.  Similarly, $F$ be an idempotent obtained from reduced column echelon form of $A$ by interchanging columns, so that leading 1's on the diagonal. There exist invertible matrices $P$ and $Q$  such that $A=PE=FQ$, 
$ann_r(A)=ann_r(E)=(I-E)R$ and $ann_l(A)=ann_l(F)=R(I-F)$.  
  Let $~T_{I-E}(X)= (I-E)X:R \longrightarrow R$ be a map. Then  $(I-E)R=range(T_{I-E})=W_1\oplus W_2\oplus\ldots\oplus W_n$, where $W_k=\left\{[C_1,C_2, \ldots, C_k,\ldots,C_n] \colon~C_i\in F_q^n,~ C_i=0,~\text{for all}~i\neq k,~~\text{and}~~ (I-E)C_k=C_k\right\}$.  Hence the dimension of $range(T_{I-E})$ is  $\displaystyle \sum_{i=1}^{n}dim(W_i)=n(n-r)$.  Let $\left\{v_1,\ldots,v_{n(n-r)} \right\}$ be a basis of $range(T_{I-E})$. Then $range(T_{I-E})=\{ k_1v_1+\ldots+k_{n(n-r)}v_{n(n-r)} \colon k_1, \ldots, k_{n(n-r)}\in F_q \}$.  Therefore $|ann_r(A)|=|(I-E)R|=|range(I-E)|=q^{n(n-r)}$. Similarly $|ann_l(A)|=q^{n(n-r)}$.\\
 Now $ann_r(A)\cap ann_l(A)=ann_r(E)\cap ann_l(F)=((I-E)R)\cap (R(I-F))=(I-E)R(I-F)$. Since $nullity(I-E)$ is $r$, its row echelon form has $r$ zero rows. Similarly column echelon form of $I-F$ has $r$ zero columns. Therefore any matrix of the form $(I-E)B(I-F)$ has $r$ zero rows and $r$ zero columns. So that it has $2rn-r^2$ zero entries and other $n^2-(2rn-r^2)$ entries  are arbitrary.  Therefore number of matrices of the form $(I-E)B(I-F)$ is $q^{(n-r)^2}$.
 Therefore $d(A)=2q^{n(n-r)}-q^{(n-r)^2}-1$.
\end{proof}
\begin{remark}
In above lemma, if we take $A^2=0$ then 
\begin{align*}
&d(A)=|N(A)|
\\&=|ann_r(A)\setminus\{A\}|+|ann_l(A)\setminus\{A\}|-|ann_r(A)\cap ann_l(A)\setminus \{A\}|-1\\
&=2q^{n(n-r)}-q^{(n-r)^2}-2
\end{align*}
where $r=rank(A)$.
\end{remark}
\begin{lemma}\label{l3}
Let $q=p^k$. The number of nontrivial idempotent matrices in $M_n(F_q)$ is $\displaystyle\sum_{r=0}^{n}q^{r(n-r)}{n\choose r}_q-2 $. Also, the number of nilpotent matrices of index 2 in $M_n(F_q)$ is $\displaystyle \sum_{r=1}^{[n/2]}{n\choose r}_{q}{n-r \choose r}_{q}$.
\end{lemma}
\begin{proof}
Let $A$ be an idempotent matrix in $M_n(F_q)$ of rank r. Hence $A$ is similar to the diagonal matrix $diag(I_r, O_{n-r})$. Consider a action of group $GL_n(F_q)$ on set\\ $S=\left\{ A_r=diag(I_r, I_n-r) \colon r=1,2, \ldots, n-1\right\}$ defined by $f(A)=PAP^{-1},~~~\text{for all} ~~A\in S$.
Hence for each $A_r\in S,$  $\displaystyle |O(A_r)|=\frac{|GL_n(F_q)|}{|N(A_r)|},~~\text{where} ~~O(A_r)~~\text{is the orbit containing}~~A_r$ and $$N(A_r)=\left\{ P\in GL_n(F_q) \colon PA_r=A_rP \right\}.$$
Now if $P\in N(A_r)$, then $P=diag(Q, R)$, where $Q\in GL_r(F_q),~R\in GL_{n-r}(F_q)$. Hence $|N(A_r)|=|GL_r(F_q)|.|GL_{n-r}(F_q)|$. Therefore
 \begin{align*} |O(A_r)|&=\frac{|GL_n(F_q)|}{|GL_r(F_q)|.|GL_{n-r}(F_q)|}\\&=\frac{\prod_{i=0}^{n-1}(q^n-q^i)}{\left(\prod_{i=0}^{r-1}(q^r-q^i)\right)\left(\prod_{i=0}^{n-r-1}(q^{n-r}-q^i)\right)}\\&={n \choose r}_{q}\frac{\prod_{i=r}^{n-1} (q^n-q^{i})}{\prod_{i=0}^{n-r-1}(q^{n-r}-q^i)}\\&=q^{r(n-r)}{n \choose r}_q.\end{align*}
 Hence the number of all nonzero idempotents is equal to $\displaystyle \sum_{r=0}^{n}|O(A_r)|-2=\sum_{r=0}^{n}q^{r(n-r)}{n \choose r}_q-2$.
 Note that,  $N\in M_n(F_q)$ is a nonzero matrix of nilpotency index 2 and of rank $r$ if and only if $(0)\subset range(N)\subseteq ker(N)\subset F_q^n$
 and $dim(range(N))=r\leq dim(ker(N))=n-r$, {\it i.e.}, $r\leq [n/2]$. Therefore number of choices for $ker(N)$ is ${n \choose n-r}_{q}$ and number of choices of $range(N)$ is ${n-r \choose r}_{q}$.  By Proposition \ref{p5}, two matrices are related under the relation $\sim$ if and only if they have the same range and the same kernel. Hence the total number of nonzero nilpotent matrices of index 2 is 
 $$\sum_{r=1}^{[n/2]}{n \choose n-r}_{q}{n-r \choose r}_{q}=\sum_{r=1}^{[n/2]}{n \choose r}_{q}{n-r \choose r}_{q}.$$
\end{proof}
\begin{lemma}\label{l4}
The number of equivalence classes of $\sim$ in $M_n(F_q)$ is  $\displaystyle \sum_{r=1}^{n-1}{n \choose r}_q^2$. 
\end{lemma}
\begin{proof}
If $A\sim B$ in $M_n(F_q)$, then $A$ and $B$ have the same rank. Let $n_r$ be the number of equivalence classes of $\sim$ in $C_r,$ where $C_r$ is a set of all rank $r$ matrices. Hence the total number of equivalence classes is $\displaystyle m=\sum_{r=1}^{n-1} n_r$. If $A$ is a matrix of rank $r$, then  by Lemma \ref{l2}, the cardinality of the equivalence class containing $A$ is $|[A]|=\prod_{i=0}^{r-1}(q^r-q^i)$. Hence $n_r=\frac{|C_r|}{\prod_{i=0}^{r-1}(q^r-q^i)}$.\\ In $[7]$, it is given that $|C_r|=\frac{\prod_{i=0}^{r-1}(q^n-q^i)^2}{(q^r-q^i)}$. Therefore $$m=\sum_{r=1}^{n-1}n_r=\sum_{r=0}^{r-1}\frac{\prod_{i=0}^{r-1}(q^n-q^i)^2}{(q^r-q^i)^2}=\sum_{r=1}^{n-1}{n \choose r}_{q}^2.$$
\end{proof}
\begin{definition}
Let $q=p^k$ and $F_q$ be a finite field.  Let $T=\left\{[A] \colon A\in Z(M_n(F_q))\right\}$ be a set of all equivalence classes  of the relation $\sim$. The directed graph $\overline{\Gamma}(T)$ is a graph with vertex set $T$ and there is a directed edge 
$[A]\longrightarrow [B]$ between two vertices $[A] ~and~ [B]$ in $T$ if and only if $AB=0$, {\it i.e.}, $range(B)\subseteq ker(A)$.
Note that there is an undirected graph $\Gamma(T)=G^{\sim}$, where 
$[A]-[B]$ is an edge in $G^{\sim}$ if and only if $AB=0~\text{or}~BA=0,$ that is, $range(B)\subseteq ker(A)$ or $range(A)\subseteq ker(B)$. 
\end{definition}
\begin{definition}
Let $F_q$ be a finite field and $$S=\left\{ (U, V)~\colon~U, V \text{are subspaces of}~~F^n_q~~\text{with}~dim(V)+dim(W)=n \right\}.$$ The directed graph $\overline{\Gamma}(S)$ is a graph on a vertex set $S$ and with an edge set defined as:\\
$(U_1, V_1)\longrightarrow (U_2, V_2)$ if and only if $U_2\subseteq V_1$.
The undirected graph ${\Gamma}(S)$ is a graph on a vertex set $S$ and with an edge set defined as:\\
$(U_1, V_1)-(U_2, V_2)$ if and only if $U_2\subseteq V_1$ or $V_2\subseteq U_1$
\end{definition}
\begin{proposition}
 Let $[A]\in T$ and rank of $A$ is $r$. Then in a graph $\overline{\Gamma}(T),$
$$\displaystyle d^+([A])=d^-([A])=\sum_{i=1}^{n-r}{n-r\choose i}_q{n\choose i}_q.$$
\end{proposition}
\begin{proof}
Define a map $\phi:\overline{\Gamma}(T) \longrightarrow\overline{\Gamma}(S)$ by $\phi([A])=(range(L_A), ker(L_A))$,  for all $[A]\in T$. Observe that $\phi$ is a graph isomorphism.\\
Number of pairs of subspaces $(U, V)$ of $F_q^n$ such that $dim(U)+dim(V)=n$ and $dim(V)=i$ is equal to 
$\displaystyle {n\choose i}_q {n\choose n-i}_q={n\choose i}_q^2$.
Hence the total number of vertices in $\overline{\Gamma}(S)$ is equal to total number of all such pairs of  subspaces $(U, V)$ such that $dim(U)+dim(V)=n$; and it is given by
$\displaystyle \sum_{i=1}^{n-1}{n\choose i}_q^2$.
In $\overline{\Gamma}(T)$, a vertex $(U',V')$ is post adjacent to $(U, V)$ if and only if
$U'\subseteq V$. Let $A\in M_n(F_q)$ with $rank(A)=r$ and $t=n-r$. Let $U=range(A),~V=ker(A)$. Therefore $d^+([A])=d^+(U, V)$. The number of subspaces of $V$ of dimension $i$ is equal to $\displaystyle{t\choose i}_q$. For each subspace $X$ of $V$ with a dimension equal to $i,$ the number of vertices of the form $(X,*)$ is equal to $\displaystyle {t\choose i}_q{n\choose n-i}_q={t\choose i}_q{n\choose i}_q$.
Hence there are  $\displaystyle \sum_{i=1}^{t}{t\choose i}_q{n\choose i}_q=\sum_{i=1}^{n-r}{n-r\choose i}_q{n\choose i}_q$ post adjacent vertices of $(U, V)$.\\
In $\overline{\Gamma}(T)$, a vertex $(U', V')$ is pre-adjacent to $(U, V)$ if and only if $U \subseteq V'$. Since $dim(U)=n-dim(V)=r,$  $F^n_q/U\equiv F^{t}_q$. Hence number of subspaces of $F^n_q$ with dimension $j$ that contains $U$  is equal to the number of subspaces of $ F^{t}_q$ having dimension $j-r,$ and this is equal to $\displaystyle {n-r\choose j-r}_q={n-r\choose n-j}_q$. Hence the number of all pre-adjacent vertices of $(U, V)$ of the form $(U', V')$ with $dim(V')=j$ is equal to $\displaystyle {n-r\choose n-j}_q{n\choose n-(n-j)}_q= {n-r\choose n-j}_q{n\choose n-j}_q$.  $$\text{Hence}~~d^-((U,V))=\sum_{j=1}^{n-r}{n-r\choose n-j}_q{n\choose n-j}_q=\sum_{i=1}^{n-r}{n-r\choose i}_q{n\choose i}_q.$$
Therefore $$\displaystyle d^+([A])=d^-([A])=\sum_{i=1}^{n-r}{n-r\choose i}_q{n\choose i}_q.$$
\end{proof}
\begin{corollary}\label{c3}
Let $q=p^k$ with $p$ prime and  $A\in Z(M_n(F_q))$.  Then 
$$d([A])= 2\sum_{i=1}^{n-r}{n-r\choose i}_q{n\choose i}_q-\sum_{i=1}^{n-r}{n-r\choose i}_q^2,$$
where $r=\text{rank}(A)$.
\end{corollary}
\begin{proof} 
Let $R=M_n(F_q)$, $A\in Z(R)$ and $rank(A)=r$. We have $d(A)=d^+(A)+d^-(A)-|S|$, where $\displaystyle S=\{[B]\in Z(R) \colon [B][A]=[A][B]=0\}=\{(X=range(B),Y=ker(B)) ~|~ X~ \text{ is subspace} \text{ of }~ ker(A)~\text{ and}~ range(A)~\text{ is subspace of }~Y \}$. 
Now we will find $|S|$. Let $U=range(A),~~V=ker(A)$. 
The number of possible pairs of subspaces  $(X,Y)$ such that $X\subseteq V,~~U\subseteq Y$ and $dim(X)=i,~dim(Y)=n-i$ is equal to
$\displaystyle {n-r\choose i}_q{n-r\choose n-r-i}_q={n-r\choose i}_q{n-r\choose i}_q$. Hence the total number of pairs of subspaces $(X,Y)$ required is\\ $\displaystyle \sum_{i=1}^{n-r}{n-r\choose i}_q{n-r\choose i}_q$.
Hence $\displaystyle|S|=\sum_{i=1}^{n-r}{n-r\choose i}_q^2$.
 $$\text{Therefore}~~d([A])=2\sum_{i=1}^{n-r}{n-r\choose i}_q{n\choose i}_q-\sum_{i=1}^{n-r}{n-r\choose i}_q^2.$$
\end{proof}
\begin{theorem}\label{t2}
 Let $q=p^k$ with $p$ prime. Consider a ring $R=M_n(F_{q})$.
 For each $A_i\in \Gamma(R),$ let 
 \begin{align*} 
 &[A_i]=\left\{ B\in Z(R) \colon  B\sim A_i ~~ie.,~B=PA_i=A_iQ~\text{for some P, Q}\in GL_n(F_q) \right\}.\\
& X=\{ [A_i] \colon A_i\in Z(R) \}
 X_2=\{[A_i]\in X \colon A_i\neq 0, A_i^2=0 \},
 Y_2=\{[A_i]\in X \colon  A_i^2=A_i \}\\
& n_i=|[A_i]|, l=|N_2|, m=|Y_2|, d_i=d([A_i]), r_i=rank(A_i),~~
\displaystyle N_{i}=\sum_{A_j\in N(A_i)}n_j.
 \end{align*}  
 Then  
 \begin{align*}
   &n_i=\prod_{k=1}^{r_i-1}(q^{r_i}-q^k),~ d_i=2\sum_{k=1}^{n-r_i}{n-r_i\choose k}_q{n\choose k}_q-\sum_{k=1}^{n-r_i}{n-r_i\choose k}^2_q,\\
& m=\sum_{k=0}^{n}q^{k(n-k)}{n\choose k}_q-2, ~~l=\sum_{k=1}^{[n/2]}{n\choose k}_q{n-k\choose k}_q,
 \end{align*}   
 \begin{enumerate} 
 \item $$\displaystyle \Gamma(R)=\bigvee_{\Gamma(R)^{\sim}}\left\{ \overline{K_{n_i}} \colon [A_i]\in X\setminus X_2\right\}\bigcup \left\{ K_{n_i} \colon [A_i]\in X_2\right\},$$ 
 $$\sigma_A(\Gamma(R))=\left(\bigcup_{[A_i]\in X_2}\left\{-1^{(n_i-1)}\right\}\right)\bigcup\left( \bigcup_{[A_i]\in X\setminus X_2}\left\{ 0^{(n_i-1)} \right\}\right) \bigcup \sigma(C_A(\Gamma(R))),$$\\
where $C_A(\Gamma(R))$ is a square matrix of order $l$ defined as below.  
 $$C_A(\Gamma(A))=(c_{ij})=\begin{cases} n_i, ~~~& [A_i]=[A_j]\\
                                 \sqrt{n_in_j}~~,& [A_i]~~ \text{adjacent to}~~ [A_j]\\
                                 0,~~& \text{otherwise}

          \end{cases}
 $$
 and
 \item $$\displaystyle\sigma_L(R)=\left(\bigcup_{[A_i]\in X_2}\left\{ (N_i+n_i)^{(n_i-1)}\right\} \right) \bigcup \left( \bigcup_{[A_i]\in X \setminus X_2}\left\{ N_i^{(n_i-1)} \right\}\right)\bigcup \sigma(C_N(\Gamma(R))),$$
 where 
 $$C_N(\Gamma(R))=(d_{ij})=\begin{cases} N_{i}, ~~~& [A_i]=[A_j]\\
                                  -\sqrt{n_{i}n_{j}},~~& [A_i]~~ \text{adjacent to}~~ [A_j]\\
                                  0,~~& \text{otherwise}

           \end{cases}
  .$$
 \end{enumerate}
 \end{theorem}
\begin{proof}
The proof follows from Proposition \ref{p1}, Lemma \ref{l2}, \ref{l3}, \ref{l4} and Corollary \ref{c3}.
\end{proof}

\section{Spectra of zero-divisor graph of finite semisimple rings}

\begin{lemma}\label{l6}
Let $I$ be an indexing set and $i \in I$. Let $R_i$ be finite ring and $T=\prod_{i\in I}R_i$. Then the following statements hold.
\begin{enumerate}
\item  Let  $x=(x_i)~~_{i\in I}, y=(y_i)_{i\in I}$ be any two elements in $Z(T)$. The relation $\sim_T$ defined by $$x\sim_T y~~\text{ if and only if}~x_i=u_iy_i=y_iv_i,~\text{ for some units}~~u_i, v_i\in R_i$$ is an equivalence relation.
Further, the relation $\sim_T$ is equivalent to the relation $\sim$ which is defined as, $$x\sim y~\text{ if and only if} ~x=uy=yv,~\text{ for some units}~u, v\in T.$$ 
\item Let $x=(x_i)_{i\in I}\in Z(T)$; and $I_1=\{ i\in I \colon x_i~~\text{is unit} \},~I_2=\{ i\in I \colon x_i=0\}, I_3=I\setminus (I_1\cup I_2). $ Then $|[x]|=\left(\prod_{\in I_1}|U(R_i)|\right)\left(\prod_{i\in I_3}|[x_i]|\right)$.
\item Let $x=(x_i)_{i\in I}\in Z(T)$ and $I_1=\{ i\in I \colon x_i\neq 0\},~~I_2=\{ i\in I \colon x_i=0\}$.
 Then $\displaystyle d(x)=\prod_{i\in I_1}(d(x_i)+1)\prod_{i\in I_2}|R_i|-1$;
 and $\displaystyle d([x])=\prod_{i\in I_1}(d([x_i])+1)\prod_{i\in I_2}|\Gamma(R_i)^{\sim}|-1$.
\end{enumerate}
\end{lemma}
\begin{proof}(1)
Let $x=(x_i)_{i\in I}$ and $y=(y_i)_{i\in I}$ in $Z(T)$.
Assume that $x\sim y$. Hence there exist units $u=(u_i)_{i\in I}$ and $v=(v)_{i\in I}$ in $T$ such that  $x=(x_i)_{i\in I}=uy=(u_iy_i)_{i\in I}=yv=(y_iv_i)_{i\in I}$. Since $u,v$ are units, $u_i, v_i$ are also units, for each $i$. Therefore $x_i\sim y_i$, for all $i\in I$. Hence $x\sim_T y$.  Similarly the converse follows.\\
(2) Let $x=(x_i)_{i\in I}\in Z(T)$ and $y=(y_i)_{i\in I}\in Z(T)$. Let $y\sim x$. Hence $y_i\sim x_i$, for all $i\in I$. Observe that $[x_i]= U(R_i)~ \text{if } ~x_i~\text{is unit}$ and $[0]= \{0\}$ in the ring $R_i$.  Now if $x_i$ is nonzero non unit, then $x_i\in Z(R_i),$ because  $R_i$ is finite ring.  Hence by the multiplication principle of counting, (2) holds.\\    
(3) Let $x=(x_i)_{i\in I}\in Z(T)$ and $I_1=\{ i\in I \colon x_i\neq 0\},~~I_2=\{ i\in I \colon x_i=0\}$.
 If $y=(y_i)_{i\in I}\in Z(T)$ such that $xy=0$, then $x_iy_i=0$, for all $i\in I$. Hence $y_i\in N(x_i)\cup\{0 \}$, for $i\in I_1$ and $y_i\in R_i$, for $i\in I_2$. Therefore $\displaystyle d(x)=\prod_{i\in I_1}(|N(x_i)|+1)\prod_{i\in I_2}|R_i|-1$, where $N(x_i)$ is a set of all neighbors of $x_i$ in a graph $\Gamma(R_i)$.
 Hence we get $\displaystyle d(x)=\prod_{i\in I_1}(d(x_i)+1)\prod_{i\in I_2}|R_i|-1$. Similarly we can prove that, $\displaystyle d([x])=\prod_{i\in I_1}(d([x_i])+1)\prod_{i\in I_2}|\Gamma(R_i)^{\sim}_i|-1$.
\end{proof}
\begin{proposition}\label{p8}
For $k\in I=\{1,2,\ldots, t\}$, let $n_k,m_k$ be positive integers. Let $p_k$ be distinct primes and $q_{k}=p_k^{m_k}$.  Let $\displaystyle R=\bigoplus_{k=1}^{t} M_{n_k}(F_{q_k})$ be a ring, where each $F_{q_k}$ is a finite field. If $\displaystyle A=(A_1, \ldots, A_t)\in R,$
 $rank(A_k)=r_k$, for all $k=1,2, \ldots, t$; and $I_1=\{k \colon r_k=n_k\}, I_2=\{k \colon r_k=0 \}, I_3=I\setminus(I_1\cup I_2)$ and $I_4=\{k \colon r_k\neq 0 \}$. Then 
\begin{enumerate}
\item $\displaystyle|[A]|=\prod_{k\in I_1} \left(\prod_{i=1}^{n_k-1}(q_k^{n_k}-q_k^i) \right)\prod_{k\in I_3}\left(\prod_{i=0}^{r_k-1} (q_k^{r_k}-q_k^i)\right)$.
\item $\displaystyle  d([A])=\prod_{k\in I_2}\left(\sum_{i=1}^{n-1}{n\choose i}_q^2\right)\prod_{k\in I_4}\left(\sum_{i=1}^{n-r_k}\left(2{n-r_k\choose i}_q{n\choose i}_q-{n-r_k\choose i}_q^2\right)\right)-1$.
\item $\displaystyle d(A)=\left(\prod_{k\in I_2}q_k^{n_k^2}\right)\left( \prod_{k\in I_4}\left(2q^{n_k(n_k-r_k)}-q^{(n_k-r_k)^2}\right) \right)-1$.
\end{enumerate}
\end{proposition}
\begin{proof}
Proof follows from Lemma \ref{l2}, \ref{l5} and Corollary \ref{c3}.
\end{proof}
\begin{theorem}
$\text{For}~~k\in I=\{1,2,\ldots, t\}$, let $m_k,n_k$ be positive integers.  
 Let $p_k$ be distinct primes and $q_{k}=p_k^{m_k}$.
  Let $\displaystyle R=\bigoplus_{k=1}^{t} M_{n_k}(F_{q_k})$ be a ring, where each $F_{q_k}$ is a finite field.\\
 For each $A_i=(A_{i1},A_{i2}, \ldots, A_{it})\in \Gamma(R),$ let  $rank(A_{ik})=r_{ik}$, for all $k=1,2, \ldots, t$,  
 \begin{align*} 
  & I_1=\{k \colon r_k=n_k\}, I_2=\{k \colon r_k=0 \}, I_3=I\setminus(I_1\cup I_2) ~~\text{and}~~~I_4=\{k \colon r_k\neq 0 \},\\
 &[A_i] =\{ B=(B_k)_{k=1}^{t}\in Z(R) \colon  B_{k}\sim A_{ik} ~k\in I \},~~
  X=\{ [A_i] \colon A_i\in \Gamma(R) \},\\
& X_2 =\{[A_i]\in X \colon A_i\neq 0, A_i^2=0 \},
 Y_2=\{[A_i]\in X \colon  A_i^2=A_i \},\\
 & n_i=|[A_i]|, l=|N_2|, m=|Y_2|, d_i=d([A_i]), r_{ik}=rank(A_{ik}),~~
 \displaystyle N_{i}=\sum_{A_j\in N(A_i)}n_j.
 \end{align*}  
 Then  
 \begin{align*}
  & n_i=\prod_{k\in I_1} \left(\prod_{j=1}^{n_k-1}(q_k^{n_k}-q_k^j) \right)\prod_{k\in I_3}\left(\prod_{j=0}^{r_{ik}-1} (q_k^{r_{ik}}-q_k^j)\right),\\ 
  & d_i=\prod_{k\in I_2}\left(\sum_{l=1}^{n_k-1}{n_k\choose l}_q^2\right)\prod_{k\in I_4}\left(\sum_{l=1}^{n_k-r_{ik}}\left(2{n_k-r_{ik}\choose l}_{q_k}{n_k\choose l}_{q_k}-{n_k-r_{ik}\choose i}_{q_k}^2\right)\right)-1,\\
& m=\prod_{k=1}^{t}\left(\sum_{j=0}^{n_k}q_k^{j(n_k-j)}{n_k\choose j}_{q_k}\right)-2,~~ l=\prod_{k=1}^{t}\left(\sum_{j=1}^{[n_k/2]}{n_k\choose j}_{q_k}{n_k-j\choose j}_{q_k}\right)
 \end{align*}   
 and  adjacency and Laplacian spectra of $\Gamma(R)$ are given as in Theorem \ref{t2}.
 \end{theorem} 
 \begin{proof}
 The proof follows from Lemma \ref{l6} and Proposition \ref{p8}.
 \end{proof}
 \begin{corollary}
Let $k\in I=\{1,2,\ldots, t\}$, $m_k$ are positive integer and $q_{k}=p_k^{m_k}$, where $p_k$ are distinct primes.  Let $\displaystyle R=\bigoplus_{k=1}^{t} M_{2}(F_{q_k})$ be a ring.
 For each $A_i=(A_{i1},A_{i2}, \ldots, A_{it})\in \Gamma(R),$ 
 \begin{align*} 
 & I_1=\{k \colon A_{ik} ~~\text{is unit} \},~~ I_2=\{k \colon A_{ik}=0 \}, ~~I_3=I\setminus (I_1\cup I_2),\\
 &[A_i] =\left\{ B=(B_k)_{k=1}^{t}\in Z(R) \colon  B_{k}\sim A_{ik} ~k\in I \right\}.~~
 X=\{ [A_i] \colon A_i\in \Gamma(R)) \},\\
& X_2=\{[A_i]\in X \colon A_i\neq 0, A_i^2=0 \},~~
 Y_2=\{[A_i]\in X \colon  A_i^2=A_i \},\\
& n_i=|[A_i]|, l=|N_2|,~~ m=|Y_2|, d_i=d([A_i]),~~ r_{ik}=rank(A_{ik}),~~
 \displaystyle N_{i}=\sum_{A_j\in N(A_i)}n_j.
 \end{align*}  
 Then  
 \begin{align*}\displaystyle
   & n_i=\prod_{k\in I_1} \left(q_k^2-q_k \right)\prod_{k\in I_3}\left( q_k-1\right),~~~
   d_i=\prod_{k\in I_2}(q_k+1)\prod_{k\in I_4}(2q_k+1)-1,\\
 & m=\prod_{k=1}^{t}\left(\sum_{j=0}^{2}q_k^{j(2-j)}{2\choose j}_{q_k}\right)-2, ~~l=\prod_{k=1}^{t}(q_k+1)
 \end{align*}   
 and  adjacency and Laplacian spectra of $\Gamma(R)$ are given as in Theorem \ref{t2}.
 \end{corollary}

\section{A method to find spectra of the generalized join of graphs}
Let $H$ be a graph on $I=\{1,2, \ldots, n\}$ vertices and for each $i\in I$, $G_i$ be a graph on $\{v_{i1}, \ldots, v_{in_i}\}$ vertices. If $\displaystyle G=\bigvee_{H}\{G_1, G_2, \ldots, G_n\}$, then  $A(G)$ is a block matrix  $\displaystyle \begin{bmatrix}A(G_1)& J_{12}& J_{13}& \cdots & J_{1n}\\J_{21}& A(G_2)& J_{23}& \cdots & J_{2n}\\ \vdots& \vdots& \vdots& \ddots & \vdots\\J_{n1}& J_{n2}& J_{n3}& \cdots & A(G_n) \end{bmatrix},$ where $J_{ij}$ is a matrix of all $1's$ if $i-j$ is an edge in $H$ and $J_{ij}$ is a matrix of all $0's$ if $i-j$ is not an edge in $H$. The order of $J_{ij}$ is $n_i\times n_j$.
If all graphs $G_i$ are null graphs, then $G=\bigvee_{H}\{\overline{K}_{n_1},\overline{K}_{n_2}, \ldots, \overline{K}_{n_n}\}$ is multipartite graph and  $A(G)=\begin{bmatrix} O_{n_1}& J_{12}& J_{13}& \cdots & J_{1n}\\J_{21}& O_{n_2}& J_{23}& \cdots & J_{2n}\\ \vdots& \vdots& \vdots& \ddots & \vdots\\J_{n1}& J_{n2}& J_{n3}& \cdots & O_{n_n} \end{bmatrix}$. 
In this case, $A(G)$ is obtained by duplicating $i^{th}$ row and $i^{th}$ column by $n_i$ times iteratively.
Now we have one important observation about the eigenvalues and eigenvectors of matrices. 
\begin{proposition}\label{p4}
Let $j\in \{ 1,3...,n\}$ and $m$ be a positive integer.
Let $B$  be a square matrix of size $n$ and $A$ be a matrix obtained by duplicating  $j^{th}$ row of $B$ $m$ times and  then duplicating $j^{th}$ column of new matrix $m$ times.
Let $v_j=[x_1, \ldots, x_{n-1}]^t$ and $w_j=[x_1, \ldots, \underbrace{x_j,x_j,x_j,\ldots,x_j}_{m-times}\ldots,x_{n-1}]^t$.\\ If $Bv_j=\lambda_j v_j$ and $Aw_j=\mu_jw_j$ then  $$\displaystyle\mu_j=\lambda_j+\frac{\displaystyle \sum_{i=1}^{n}a_{ij}}{\displaystyle \sum_{i=1}^{n}x_i}(m-1)x_j.$$    
\end{proposition}
\begin{proof}
Let $\displaystyle B=[a_{ij}]_{n\times n}$ be a matrix of size $n\times n$. If $[x_1, \ldots, x_n]^t$  is an eigenvector of $B$ corresponding to an eigenvalue $\lambda$, then we have
$$a_{i1}x_1+\ldots+a_{ij}x_j+\ldots+a_{in}x_n=\lambda x_i,~~~~~~\text{for all}~~i=1,2, \ldots, n.$$
If $[x_1, \ldots, x_{j1}=x_j, \ldots, x_{jm}, \ldots, x_{n}]^t$ is an eigenvector of $A$ associated to its eigenvalue  $\mu$, then
we have 
\begin{equation}\label{e2}
a_{i1}x_1+\ldots+a_{ij}(x_{j1}+\ldots+x_{jm})+\ldots+a_{in}x_n=\mu x_i,~~~~~~\text{for all}~~i=1,2, \ldots, n
\end{equation}
and
\begin{equation}
a_{j1}x_1+\ldots+a_{jk}(x_{j1}+\ldots+x_{jm})+\ldots+a_{kn}x_{n}=\mu x_{jk},~~\text{ for all}~ k=1, \ldots, m.
\end{equation} Therefore we get $$a_{ij}\left(\sum_{k=2}^{m}x_{jk}˘\right)=(\mu-\lambda)x_i,~~\text{for all}~i=1,2,\ldots,n$$ and $$ a_{jk}\left(\sum_{k=2}^{m}x_{jk}\right)=\mu x_{jk}-\lambda x_j,~\text{ for all}~~ k=1,2,\ldots, m.$$
Hence we have, $$\left(\sum_{i=1}^{n}a_{ij}\right) \left( \sum_{k=2}^{m}x_{jk}\right)=(\mu-\lambda)\left(\sum_{i=1}^{n}x_i\right) ~~~\text{and}~~~0=\mu(x_{jk}-x_j),~~~\text{for}~~k=2,\ldots,m.$$
If  $\mu\neq 0$, then $x_{jk}=x_j,~~\text{for}~k=2, \ldots, m$. If $\displaystyle\sum_{i=1}^{n}a_{ij}\neq 0$, then 
 $\displaystyle \mu=\lambda+\frac{\sum_{i=1}^{n}a_{ij}}{\sum_{i=1}^{n}x_i}(m-1)x_j$.
 Clearly, the last part of the statement follows from equation (\ref{e2}).
\end{proof}
We discuss above proposition  by an example.
Consider a $3\times 3$ matrix, $B=\begin{bmatrix}-1&0& 1\\0& 2& 0\\0&0&1  \end{bmatrix}$. Its eigenvalues and corresponding eigenvectors are $\lambda_1=-1,\lambda_2=2,\lambda_3=1$ and $v_1=\begin{bmatrix}1& 0& 0\end{bmatrix}^t, v_2=\begin{bmatrix}0& 1& 0 \end{bmatrix}^t, v_3=\begin{bmatrix}1&0&2\end{bmatrix}^t$ respectively. Let us obtain matrix $A,$  by duplicating second row and second column of $B$, so
$A=\begin{bmatrix}-1&0&0& 1\\0& 2&2& 0\\0& 2& 2& 0 \\0& 0 & 0& 1  \end{bmatrix}$. Now if we duplicate the second entry of $v_2$ and construct $w_2=\begin{bmatrix}0& 1 & 1& 0 \end{bmatrix}^t$, then $w_2$ is eigenvector of $A$ with associated eigenvalue $\mu_2=4=\lambda_2+\frac{0+2+0}{0+1+0}1$. Also $\lambda_1=-1$ and $\lambda_3=1$ are again eigenvalues of $A$ with corresponding eigenvectors $w_1=\begin{bmatrix}1& 0 & 0&0 \end{bmatrix}^t$ and $w_3=\begin{bmatrix}1& 0 & 0&2 \end{bmatrix}^t$ respectively.
\begin{proposition}
Let $R$ be a finite ring with unity. Let $\mathcal{F}=\{[v_i] \colon i=1,2,\ldots,k\}$ be a set of all distinct equivalence classes on $Z(R)$ with respect to the relation $\approx$ and    $|[v_i]|=n_i,~~\text{for}~~i=1,2,\ldots,k$. 
Let $ [x_1,x_2, \ldots, x_k]$ and  $\displaystyle y=[\underbrace{x_1, \ldots, x_1}_{n_1-times}, \ldots, \underbrace{x_j, \ldots, x_j,}_{n_j-times}\ldots,\underbrace{x_{k}, \ldots, x_{k}}_{n_k-times}]^t $.\\
 If  $A(\Gamma(G)^{\approx})x=\lambda x$  and   $A(\Gamma(R))y =\mu y$ then 
  \begin{align*}
 \displaystyle \mu &=\lambda+\frac{\displaystyle\sum_{i=1}^{k}a_{i1}}{\displaystyle\sum_{i=1}^{k}x_i}(n_1-1)x_1+\frac{n_1a_{12}+\displaystyle\sum_{i=n_1+1}^{k}a_{i2}}{n_1x_1+\displaystyle\sum_{i=2}^{k}x_i}(n_2-1)x_2\\ &+\frac{n_1a_{12}+n_2a_{23}+\displaystyle\sum_{i=n_1+n_2+1}^{k}a_{i3}}{n_1x_1+n_2x_2+\displaystyle\sum_{i=3}^{k}x_i}(n_3-1)x_3+\ldots\\
 &+\frac{n_1a_{12}+n_2a_{23}+\ldots+n_{k-1}a_{k-1,k}+a_{k,k}}{n_1x_1+n_2x_2+\ldots+n_{k-1}x_{k-1}+x_k}(n_k-1)x_k
\end{align*}
\\
 \end{proposition}
\begin{proof}
Let $A_1(\Gamma(R)^{\approx})$ be the matrix obtained by duplicating first row and first column of $A(\Gamma(R)^{\approx})$,  $n_1$ times. Let $A_i(\Gamma(R)^{\approx})$ be the matrix obtained by duplicating $i^{th}$  row and $i^{th}$ column of $A_{i-1}(\Gamma(R)^{\approx})$,  $n_i$ times, for  $i=2,3,\ldots,k$. Using Proposition \ref{p4}, we can obtain eigenvalue $\lambda_i$ and eigenvector $y_i$ of $A_i(G)$ from eigenvalue $\lambda_{i-1}$ and eigenvector $y_{i-1}$. Hence the expressions for eigenvalue $\mu=\lambda_k$ and eigenvector $y=y_k$ of $A(\Gamma)=A_k(G^\sim)$ follows.  
\end{proof}
Let $R$ be a finite ring with unity. Suppose $u_1,u_2, \ldots, u_n$ are linearly independent eigenvectors of $A(\Gamma(R)^{\approx})$ associated to eigenvalues $\lambda_1, \ldots, \lambda_n$ of $A(\Gamma(R)^{\approx})$. Then we can find eigenvalues and eigenbasis of $A(\Gamma(R))$ by Proposition \ref{p4}.

\end{document}